\documentclass[12pt]{article}
\usepackage[margin=1in]{geometry}
\usepackage{amsmath}
\usepackage{mathdots}
\usepackage{physics}
\usepackage{amssymb}
\usepackage{graphicx}
\usepackage{longtable}
\usepackage{setspace}
\usepackage{tikz}
\usetikzlibrary{decorations.pathreplacing,calligraphy}
\graphicspath{ {./images/} }
\usepackage{centernot}
\usepackage{amsthm}
\usepackage{fullpage}
\usepackage{amsmath,amssymb,amsfonts}
\usepackage{braket}
\usepackage{mathtools}
\usepackage{xcolor}
\usepackage{tabularx}
\usepackage{natbib}
\usepackage{comment}
\usepackage{tikz-cd}
\usepackage{mleftright}
\usepackage{setspace}

\usepackage{hyperref}
\hypersetup{
    colorlinks,
    citecolor=black,
    filecolor=black,
    linkcolor=black,
    urlcolor=black
}

\newtheorem{thm}{Theorem}[section]
\newtheorem{cor}[thm]{Corollary}
\newtheorem{lem}[thm]{Lemma}
\newtheorem{prop}[thm]{Proposition}

\theoremstyle{definition}

\newtheorem{rmk}[thm]{Remark}

\newtheorem{ex}[thm]{Example}

\DeclareFontFamily{U}{wncy}{}
\DeclareFontShape{U}{wncy}{m}{n}{<->wncyr10}{}
\DeclareSymbolFont{mcy}{U}{wncy}{m}{n}
\DeclareMathSymbol{\Sha}{\mathord}{mcy}{"58}

\DeclareMathOperator{\lf}{f}

\DeclareMathOperator{\Gal}{Gal}

\DeclareMathOperator{\tors}{tors}

\DeclareMathOperator{\Cl}{Cl}

\DeclareMathOperator{\St}{St}
\DeclareMathOperator{\rad}{rad}

\newcommand{\Z}{\mathbb Z}
\newcommand{\R}{\mathbb R}
\newcommand{\N}{\mathbb N}
\newcommand{\Q}{\mathbb Q}
\newcommand{\C}{\mathbb C}

\renewcommand{\O}{\mathcal O}
\newcommand{\mf}{\mathfrak}

\newcommand{\seq}{\subseteq}
\newcommand{\wo}{\backslash}

\newcommand{\mc}{\mathcal}

\newcommand{\res}{\mathop{\mathrm{Res}}}

\title{Equidistribution of realizable Steinitz classes for cyclic Kummer extensions}
\author{Brody Lynch}
\date{University of Massachusetts \\ \today}

\begin{document}
\maketitle
\begin{abstract}
Let $\ell$ be prime, and $K$ be a number field containing the $\ell$-th roots of unity. We use classical algebraic number theory and some analytic techniques to prove that the Steinitz classes of $\Z/\ell\Z$ extensions of $K$ ordered by relative discriminant are equidistributed among realizable classes in the ideal class group of $K$. For $\ell = 2$, this was proved by Kable and Wright using the deep theory of prehomogeneous vector spaces. Foster proved that Steinitz classes are uniformly distributed between realizable classes for tamely ramified elementary-$m$ extensions using the theory of Galois modules; our approach eliminates this tameness hypothesis. 
\end{abstract}
\tableofcontents

\newpage
\section{Introduction}
Let $K$ be a number field.  It is well-known that its ring of integers $\O_K$ is a free $\Z$-module of rank $[K:\Q]$; equivalently, $\O_K$ has a $\Z$-basis.  Now, suppose $L/K$
is a finite extension of number fields.  If $\O_K$ is a PID, then $\O_L$ also has a
free $\O_K$-basis, but otherwise it may not. The obstruction to $\O_L$ being a free $\O_K$-module is measured by an ideal class of $\O_K$ called the Steinitz class of the extension $L/K$.
In this paper, we study how the Steinitz class distributes among the class group of $K$ for $L/K$ with
a given prime-order Galois group. Before we state our main results, we quickly review some definitions and recall a couple of well-known results.

Let $M$ be a finitely-generated module over a Dedekind domain $A$. Then 
\begin{equation}
\label{Steinitz Decomposition}
    M \cong A^r \oplus \mf a \oplus M_{tors}
\end{equation}
where $\mf a$ is an ideal of $A$ (see \cite[III]{LangA}). Applying (\ref{Steinitz Decomposition}) to the case when $A = \O_K$ and $M = \O_L$, the Steinitz class of the extension $L/K$ is defined as the ideal class of $\mf a$. In \cite{Artin}, Artin showed that $\O_L$ has a relative integral basis over $\O_K$ if and only if the Steinitz class of $\O_L$ over $\O_K$ is trivial.

For a number field $K$ and group $G$, not every ideal class may be the Steinitz class of an extension of $K$ with Galois group $G$. This is shown for $\Z/m\Z$ extensions for $m > 3$ in \cite{Mc}.
For a fixed $K$ and $G$, we use the term ``realizable" to refer to the ideal classes that are Steinitz classes of an extension of $K$ with Galois group $G$. Bekyel showed the existence of an asymptotic formula counting extensions with a given finite abelian Galois group and a given realizable Steinitz class up to squares. Throughout this paper, we will use $N$ to denote the relative ideal norm.

\begin{thm}[Bekyel \cite{Bekyel}]
Let $K$ be a number field and $G$ be a finite abelian group.
Also, let $\mc C \in \Cl(K)$ be the square of a realizable ideal class. Then 
then there exists a positive number $c(K, G, \mc C)$ such that
$$\#\{L/K : \Gal(L/K) \cong G,\ N(\Delta_{L/K}) \leq X,\ \St(\O_L)^2 = \mc C\} \sim c(K, G, \mc C)X^a \log X^b$$
as $X \rightarrow \infty$ for some constants $a$ and $b$ depending on $G$ and $K$.
\end{thm}

 In light of Artin's result, it is natural to investigate what proportion of extensions of some specified Galois group have trivial Steinitz class for a given base field. Relatedly, we might ask in which cases the asymptotic constant in Bekyel's theorem is independent of $\mc C$ and if we can prove such an asymptotic result in terms of $\St(\O_L)$ instead of $\St(\O_L)^2$. An answer to these questions for $S_2$ and $S_3$ extensions was given by Kable and Wright and later for $S_4$ and $S_5$ extensions by Bhargava, Shankar, and Wang.\label{N}

\begin{thm}[Kable-Wright \cite{KW}, Bhargava-Shankar-Wang \cite{Bhargava}]
\label{KW}
Let $K$ be a number field of class number $h_K$ and let $\mc C$ be an ideal class of $K$. If $L$ is a finite
extension of $K$ then denote by $\Delta_{L/K}$ the relative discriminant of $L$ over $K$. For $n = 2, 3, 4, 5,$ we have
$$\lim_{X\rightarrow \infty} \frac{\#\{L/K : \Gal(L/K)\cong S_n,\ N(\Delta_{L/K}) \leq X,\ \St(\O_L) = \mc C\}}{\#\{L/K : \Gal(L/K)\cong S_n,\ N(\Delta_{L/K}) \leq X\}} = \frac 1 {h_K}.$$
\end{thm}

For $n=2,3,5$, this result is sufficient to show equidistribution of Steinitz classes over all degree $n$ extensions because $S_n$ extensions have density 1 in all degree $n$ extensions (as ordered by field discriminant). However, for $n=4$, this density statement is false: Bhargava shows that $D_4$ extensions have a positive density in all degree 4 extensions \cite{BQuart}.

Only counting tamely ramified extensions, Kurt Foster showed equidistribution among realizable classes of elementary abelian $m$-extensions for a given positive integer $m$.

\begin{thm}[Foster \cite{Foster}]
Let $K$ be a number field, $m$ be a positive integer and $G$ be an elementary abelian $m$-group. Also let $\mc C \in R_K$ be a realizable ideal class. Then 
$$\lim_{\rightarrow \infty} \frac{\#\{L/K : \Gal(L/K) \cong G,\ L/K \text{ tame},\  N(\Delta_{L/K}) \leq X,\ \St(\O_L) = \mc C\}}{\#\{L/K :  \Gal(L/K) \cong G, \ L/K \text{ tame},\ N(\Delta_{L/K}) \leq X\}} = \frac 1 {\#R_K}.$$
\end{thm}

Along similar lines, Bekyel gave an equidistribution result for realizable classes for Kummer extensions up to squares.
\begin{thm}[Bekyel \cite{Bekyel}]
    Let $\ell \geq 3$ be prime and $K$ be a number field containing the $\ell$-th roots of unity. Let $\Cl(K)^{\ell-1}$ be the subgroup of $(\ell-1)$-powers in $\Cl(K)$ and let $\mc C \in \Cl(K)^{\ell-1}$. Then 
    $$\frac{\#\{L/K : \Gal(L/K) \cong \Z/\ell\Z, N(\Delta_{L/K}) \leq X, \St(\O_L)^2 = \mc C\}}{\#\{L/K : \Gal(L/K) \cong \Z/\ell\Z, N(\Delta_{L/K}) \leq X\}} = \frac 1 {\#\Cl(K)^{\ell-1}}.$$
\end{thm}
Bekyel noted that it would be ideal if this result could be improved to count Steinitz classes instead of squares of Steinitz classes. That is the content of our main result.

There are many different methods to approach problems of this nature. Bekyel uses local methods and Dirichlet series to prove her result. Kable-Wright and Bhargava-Shankar-Wang used the deep theory of prehomogeneous vector spaces. Foster's work was built on the theory of Galois modules, and the tamely ramified hypothesis was crucial to his argument. Agboola \cite{Agboola} extended Foster's approach to all finite abelian extensions, but gave an equidistribution result in terms of the product of ramifying primes instead of the discriminant, still under the assumption of tame ramification. 

In this paper, we use methods from classical algebraic number theory to study the equidistribution of Steinitz classes in $\Z/\ell\Z$ extensions of a number field containing the $\ell$-th roots of unity, which we will call $\ell$-Kummer extensions. This gives a new and simpler proof of Kable and Wright's theorem for $n=2$ and also proves a stronger form of Foster's theorem for cyclic $\ell$-extensions, removing the assumption that $L/K$ be tamely ramified. Furthermore, our argument is constructive in the sense that our proof gives an explicit procedure for writing down a $\gamma \in \O_K$ for each $\Z/\ell\Z$ extension $K(\sqrt[\ell]\gamma)$ of bounded discriminant with a given Steinitz class. Here is the statement of our main theorem.

\begin{thm}
\label{Main Theorem}
    Let $\ell$ be prime and $K$ be a number field containing all $\ell$-th roots of unity. We define $R_K$ to be the subgroup of $\left(\frac{\ell-1}{2}\right)$-powers in $\Cl(K)$ for $\ell > 2$ and $R_K = \Cl(K)$ for $\ell = 2$. Then for any $\mc C \in R_K$,
    $$\lim_{X\rightarrow \infty} \frac{\#\{L/K : \Gal(L/K) \cong \Z/\ell\Z,\ N(\Delta_{L/K}) \leq X,\ \St(\O_L) = \mc C\}}{\#\{L/K : \Gal(L/K) \cong \Z/\ell\Z,\ N(\Delta_{L/K}) \leq X\}} = \frac 1 {\#R_K}.$$
\end{thm}
This is a meaningful improvement on Foster's result because a positive density of extensions with a given Galois group have wild ramification. It is also an improvement of of Bekyel's result as we have replaced $\St(\O_L)^2$ with $\St(\O_L)$. To the best of our knowledge, this is the first complete equidistribution result for Kummer extensions of any prime degree.

We now give an outline of the paper. In Section \ref{SFRQ}, we will first give a formula for determining the Steinitz class of an extension $K(\sqrt[\ell]{\gamma})$ over $K$ in terms of $\gamma$. We will also outline our method of proving the main equidistribution result. In Section \ref{PK}, we justify the validity of this method. This involves giving some preliminary results on relative discriminants, on parametrization of $\Z/\ell\Z$ extensions, and on counting $\ell$-power-free ideals satisfying certain ideal class conditions. We prove the Steinitz class formula in Section \ref{FS}. In Section \ref{KE}, we give the proof of equidistribution of Steinitz classes. 

Our methods can be extended to extensions with more general abelian Galois groups. We are currently investigating similar equidistribution results for  elementary-$\ell$ extensions and for cyclic $m$-extensions for $m$ not necessarily prime. We will report our results in future papers.

\section{Outline of proof}
\label{SFRQ}
We start this section by giving a formula for the Steinitz class of an $\ell$-Kummer extension $K(\sqrt[\ell]{\gamma})$. We will give the proof in Section \ref{FS}. Before we state the formula, we describe some general notation that will persist throughout the paper. Given $\gamma \in \O_K$, we write
$$\gamma\O_K = \mf Q \prod_{j=1}^k \mf p_j^{a_j}$$
where $\mf Q$ is an ideal all of whose prime factors divide $\ell \O_K$ and $\mf p_i$ are prime ideals that do not divide $\ell\O_K$. Letting $a_j = \ell b_j + r_j$ with $0\leq r_j \leq \ell$, we define
$$\mf R = \prod_{j=1}^k \mf p_j^{b_j}$$
and
$$\mf I_i = \prod_{r_j = i} \mf p_j.$$
Then
$$\gamma\O_K = \mf Q \mf R^{\ell} \prod_{i=1}^{\ell-1} \mf I_i^i.$$
We will call $\mf R^{\ell}$ the ``$\ell$-power-part" of $\gamma\O_K$, $\mf Q$ the ``$\ell$-factor" of $\gamma\O_K$, and $\mf I_i^i$ the ``$i$-power-part of $\gamma\O_K$" ($\mf I_1$ may also be called the squarefree part). Collectively, $\mf R^{\ell} \prod_{i=1}^{\ell-1} \mf I_i^i$ will be called the ``prime-to-$\ell$-factor". In particular, for notational convenience, we will let $\mf L_{L/k}$ be the $\ell$-factor of $\Delta_{L/K}$ and $\mf F_{L/K}$ be the prime-to-$\ell$-factor of $\Delta_{L/K}$. We also let $B$ be the set of possible values for $\mf L_{L/K}$ of $\ell$-Kummer extensions $L/K$. Finally, we let $U = \O_K^\times$. \label{F} \label{L} \label{U}

\begin{prop}
\label{Steinitz Formula Prop}
    Let $\ell$ be prime and $K$ be a number field containing the $\ell$-roots of unity. We also define $[\mf A]$ to be the ideal class of the ideal $\mf A$. Then for any $\ell$-Kummer extension $L = K(\sqrt[\ell]{\gamma})$ for $\gamma \in \O_K$,
    \begin{equation}
    \label{Steinitz Formula}
    \St(\O_L) = \left[\sqrt{\frac{\mf L_{L/K}}{\left(\mf Q\mf R^\ell \prod_{i=2}^{\ell-1} \mf I_i^{i-1}\right)^{\ell-1}}}\right].
    \end{equation}
where $\mf Q$ is the $\ell$-factor of $\gamma\O_K$, $\mf R^\ell$ is the $\ell$-power-part of $\gamma\O_K$ and
$\mf I_i^i$ is  the $i$-power-part of $\gamma\O_K$ for $1\leq i \leq \ell-1$.
\end{prop}

\begin{rmk}
   The fractional ideal inside the square root on the right-hand-side will soon be shown to be a square. Therefore, it has a unique square root. The ideal class of this square root is the Steinitz class.
\end{rmk}

Note that the only term in this formula that is not a factor of $\gamma\O_K$ is $\mf L_{L/K}$. In Section \ref{PK}, we will discuss how to compute $\mf L_{L/K}$ and show that $B$ is finite. In Section \ref{KE}, we will prove that for a fixed $\mf B \in B$, there is equidistribution of Steinitz classes for $\ell$-Kummer extensions with $\mf L_{L/K} = \mf B$. Then, summing over all finitely many $\mf B \in B$ will give our desired equidistribution result. \label{mfQ}

\begin{figure}
    \centering
    \begin{tikzcd}
    \draw[step=2.0,black,thin] (0,0) grid (6,4);
    \draw (0,3)--(6,3);

    \begin{scope}[thick,decoration={calligraphic brace, amplitude=8pt}]
  \draw[decorate](-0.2,0) -- (-0.2,2) node(Q1) [midway, xshift=-2em] {\mf B_1};
  \draw[decorate](-0.2,2) -- (-0.2,3)  node(Q2) [midway, xshift=-2em] {\vdots};
  \draw[decorate](-0.2,3) -- (-0.2,4)  node(Q3) [midway, xshift=-2em] {\mf B_m};
  \draw[decorate](2,-0.2)--(0,-0.2) node(C1) [midway, yshift=-2em] {\mc C_1}; 
  \draw[decorate](4,-0.2)--(2,-0.2) node(C2) [midway, yshift=-2em] {\mc \dots}; 
  \draw[decorate](6,-0.2)--(4,-0.2) node(C3) [midway, yshift=-2em] {\mc C_n};
 \end{scope} 
\end{tikzcd}
    \caption{$\ell$-Kummer extensions $L/K$ partitioned by possible values of $\mf L_{L/K}$ and by $\St(\O_L)$}
    \label{Rectangle}
\end{figure}
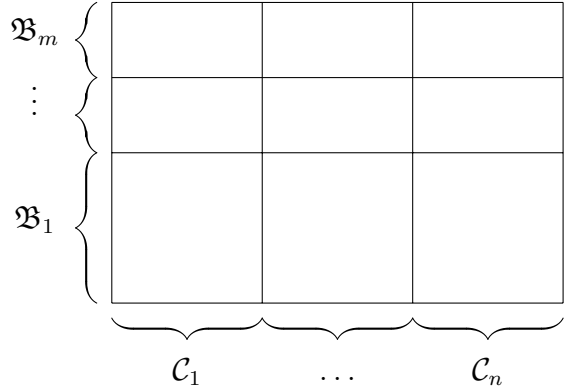

We now outline how we will find the density $\ell$-Kummer extensions with a given Steinitz class. We can envision the set of $\ell$-Kummer extensions with $N(\Delta_{L/K}) \leq X$ as partitioned in the rectangle in Figure \ref{Rectangle}. We can partition the rectangle into rows and columns, where the rows represent the $\mf L_{L/K}$ of the extensions $L/K$, and the columns represent the Steinitz class of the extensions. The lengths and widths of the rows and columns represent the density of extensions with a given $\mf L_{L/K}$ and $\St(\O_L)$ respectively. To justify and use this partitioning, we will prove the following results: 

\begin{enumerate}
    \item In a given column $\mc C_j$, each row $\mf B_i$ has a density $\rho^{\mc C_j}_{\mf B_i}$ (Proposition \ref{l-Density}).
    \item The density $\rho^{\mc C_j}_{\mf B_i}$ is independent of $\mc C_j$ (Corollary \ref{SteinitzIndependence}).
    \item For a fixed row $\mf B_i$, each column has density $\frac{1}{\#R_K}$ (Theorem \ref{Main Theorem}).
\end{enumerate} 

To be more explicit, for a given number field $K$, a prime $\ell$, a possible value $\mf B$ of $\mf L_{L/K}$ and an ideal class $\mc C \in R_K$, we define
\begin{align*}
E(X) &= \{L/K : \Gal(L/K) \cong \Z/\ell\Z,\ N(\Delta_{L/K}) \leq X\},\\
E^{\mc C}(X) &= \{L/K \in E(X) : \St(\O_L) = \mc C\},\\
E_{\mf B}(X) &= \{L/K \in E(X) : \mf L_{L/K} = \mf B\},\\
E_{\mf B}^{\mc C}(X) &= E_{\mf B}(X) \cap E^{\mc C}(X).
\end{align*}
\label{E}
From these definitions, we have that
\begin{equation}
\label{SumLPart}
    \#E^{\mc C}(X) = \sum_{\mf B \in B} \#E_{\mf B}^{\mc C}(X).
\end{equation}
Therefore, if we can show that for every $\mf B \in B$,
$\#E_{\mf B}^{\mc C}(X)$ is independent of the choice of $\mc C \in R_K$, we can get the desired equidistribution result by summing over $\mf B \in B$.

To compute $\#E_{\mf B}^{\mc C}(X)$ we must parametrize $\ell$-Kummer extensions $K(\sqrt[\ell]{\gamma})$ (Corollary \ref{Choices}). Then for each $\mf B_i \in B$, we must compute $\rho_{\mf B_i}$, the proportion of extensions that have $\mf L_{L/K} = \mf B_i$ or the length of the rows in Figure \ref{Rectangle}. More explicitly, we will show that for $\mf B \in B$ and realizable ideal class $\mc C \in R_K$,
\begin{equation}
    \label{rho_q}
    \rho_{\mf B}^{\mc C} = \lim_{X\rightarrow \infty} \frac{\#\{L/K: \Gal(L/K) \cong \Z/\ell\Z,\ N(\mf F_{L/K}) \leq X^{\ell-1},\ \St(\O_L) = \mc C,\ \mf L_{L/K} = \mf B\}}{\#\{L/K: \Gal(L/K) \cong \Z/\ell\Z,\  N(\mf F_{L/K}) \leq X^{\ell-1},\ \St(\O_L) = \mc C\}}
\end{equation} exists in Proposition \ref{l-Density}. Note that we use $X^{\ell-1}$ instead of $X$ as in \cite{CohenDisc} for convenience. We also will show that $\rho_{\mf B}$ is independent of Steinitz class, which is to say for any realizable ideal class $\mc C \in R_K$, \begin{equation}
    \label{rho_q indep}
     \lim_{X\rightarrow \infty} \frac{\#\{L/K: \Gal(L/K) \cong \Z/\ell\Z,\ N(\mf F_{L/K}) \leq X^{\ell-1},\ \mf L_{L/K} = \mf B\}}{\#\{L/K: \Gal(L/K) \cong \Z/\ell\Z,\ N(\mf F_{L/K}) \leq X^{\ell-1}\}} = \rho_{\mf B}^{\mc C}
\end{equation}

In light of this result, we will generally use $\rho_{\mf B}$ without specifying a $\mc C$. We will use $\rho_{\mf B}$ to compute $\#E_{\mf B}(X)$ and then use \eqref{rho_q indep} to restrict this computation to find $\#E_{\mf B}^{\mc C}(X)$. This will show that $\#E_{\mf B}^{\mc C}(X) = \frac {\#E_{\mf B}(X)} {\#R_K}.$

\section{Preliminary results on $\ell$-Kummer extensions}
\label{PK}
In this section, we state and prove results on $\ell$-Kummer extensions that will be useful throughout the paper. Specifically, we will state a formula for the factorization of the relative discriminant of an $\ell$-Kummer extension (Theorem \ref{Daberkow General}), give a parametrization of $\ell$-Kummer extensions (Proposition \ref{Kummer Isom Class}) and describe how they correspond to choices of $\gamma \in \O_K$ (Corollary \ref{Choices}), compute the density of $\ell$-Kummer extensions with a given Steinitz class and $\ell$-factor of the discriminant (Proposition \ref{l-Density}), and argue that this density is independent of the Steinitz class (Corollary \ref{SteinitzIndependence}).

\subsection{Parametrizing $\ell$-Kummer extensions}

First, we recall a result on computing the $\ell$-factor and prime-to-$\ell$-factor of the discriminant of an $\ell$-Kummer extension.

\begin{thm}[Daberkow \cite{D}]
    \label{Daberkow General}
    Let $K$ be a number field containing the $\ell$-roots of unity. For $\gamma \in \O_K$ let $L = K(\sqrt[\ell]{\gamma}).$
    For any prime ideal $\mf p$ not dividing $\ell\O_K$,
    \begin{enumerate}
        \item  If $\nu_{\mf p}(\gamma) \not \equiv 0 \pmod \ell$, then $\nu_{\mf p}(\Delta_{L/K}) = \ell - 1.$

        \item If $\nu_{\mf p}(\gamma) \equiv 0 \pmod \ell$, then $\nu_{\mf p}(\Delta_{L/K}) = 0.$
    \end{enumerate}
    For any prime ideal $\mf q$ dividing $\ell\O_K$,
    
    \begin{enumerate}
    \item If $\nu_{\mf q}(\gamma) \not \equiv 0 \pmod \ell$, then $\nu_{\mf q}(\Delta_{L/K}) = (\ell - 1)+ \ell \nu_{\mf q}(\ell).$

    \item If $\nu_{\mf q}(\gamma) \equiv 0 \pmod \ell$, then let 
    $$s = \max\{0 < m \leq \ell \nu_{\mf q}(\ell) : \exists c \in \O_K, c^\ell \equiv \gamma \pmod{\mf q^m}\}.$$
    If $s = \ell \nu_{\mf q}(\ell)$, then $\nu_{\mf q}(\Delta_{L/K}) = 0$. Otherwise, $\nu_{\mf q}(\Delta_{L/K}) = (\ell-1)(\ell \nu_{\mf q}(\ell) - s + 1).$
    \end{enumerate}
\end{thm}

We also prove the following elementary result for lack of a convenient reference.

\begin{lem}
    \label{Reduction General} Let $K$ be a number field and $\mf P$ be an $\ell$-power-free ideal of $\O_K$. Let $\gamma_1\O_K = \mf I_1^\ell  \mf P$ and $\gamma_2\O_K = \mf I_2^\ell \mf P$, where $\mf I_1$ and $\mf I_2$ are ideals of $\O_K$ in the same ideal class. Then $K(\sqrt[\ell]{\gamma_1})=K(\sqrt[\ell]{\alpha \gamma_2})$ for some $\alpha \in U$.
\end{lem}

\begin{proof}
    Since $\mf I_1$ and $\mf I_2$ are in the same ideal class, there exists some $k \in K^\times$ such that $\mf I_1 = k \mf I_2$. Therefore, $\gamma_1\O_K = k^\ell \gamma_2\O_K$ so for the correct choice of unit in $U$, $\gamma_1 = \alpha k^\ell \gamma_2$. Therefore,
    $K(\sqrt[\ell]{\gamma_1}) =K(\sqrt[\ell]{\alpha k^\ell \gamma_2}) = K(\sqrt[\ell]{\alpha \gamma_2}).$
\end{proof}

Combining Theorem \ref{Daberkow General} and Lemma \ref{Reduction General}, we get the following result on isomorphism classes of $\ell$-Kummer extensions. To simplify the notation, for an ideal $\mf A$, we define $\lf(\mf A)$ to be the unique $\ell$-power-free part of $\mf A$ and use $[\mf A]$ to denote the ideal class of $\mf A$. \label{f}

\begin{prop}
\label{Kummer Isom Class}
    Let $\gamma_1$ and $\gamma_2$ be elements of a number field $K$ containing the $\ell$ roots of unity.
    There is a $K$-isomorphism $\phi: K(\sqrt[\ell]{\gamma_1}) \rightarrow K(\sqrt[\ell]{\alpha\gamma_2})$ for some $\alpha \in U/U^\ell$ if and only if $\lf(\gamma_1^m\O_K) = \lf(\gamma_2\O_K)$ and $\left[\frac{\gamma_1^m\O_K}{\lf(\gamma_1^m\O_K)}\right] = \left[\frac{\gamma_2\O_K}{\lf(\gamma_2\O_K)}\right]$ for some $1 \leq m \leq \ell-1$.
\end{prop}

\begin{proof}
    Lemma \ref{Reduction General} proves the ``if" direction. All that remains is to prove the ``only if direction". If there is such a $K$-isomorphism $\phi$, then $\phi(\gamma_2) = \gamma_2$ must be an $\ell$-power in $K(\sqrt[\ell]{\gamma_1})$. This means that $\gamma_2\O_K = k^\ell \gamma_1^m\O_K$ for some $k \in K$ and $1\leq m \leq \ell-1$. By unique factorization of ideals, $\gamma_1^m\O_K$ and $\gamma_2\O_K$ must have identical $\ell$-power-free parts. Additionally, dividing out by the $\ell$-power-free part, we get $k^\ell \frac{\gamma_1^m\O_K}{\lf(\gamma_1^m\O_K)} = \frac{\gamma_2\O_K}{\lf(\gamma_2\O_K)}$ so $\left[\frac{\gamma_1^m\O_K}{\lf(\gamma_1^m\O_K)}\right] = \left[\frac{\gamma_1\O_K}{\lf(\gamma_1\O_K)}\right]$.
\end{proof}

For ease of exposition, let $P$ be the set of non-zero principal ideals in $\O_K$ and $U$ be the set of units in $\O_K$. We define $F: P \times U \rightarrow \O_K\wo\{0\}$ to be a bijection such that $F(\mf A, u) = \gamma$ where $\gamma\O_K = \mf A$ and $F(\mf A, u_1) u_2 = F(\mf A, u_2) u_1$. Explicitly, for each principal ideal $\mf A \in P$, we fix a generator $\alpha_{\mf A}$. Then $F(\mf A, u) = \alpha_{\mf A}u$ and the inverse of $F$ is given by $\gamma \in \O_K\wo\{0\} \mapsto \left(\gamma\O_K, \frac{\alpha_{\gamma\O_K}}{\gamma}\right).$ For our counting purposes, we do not care which map we pick as long as we have a fixed map $F$ that does this. Later in this section, we will describe $F$ more specifically. \label{P}

\begin{cor}
\label{Choices}
Let $A = U/U^\ell \times \Cl(K) \times Q \times J$ where $Q$ is the set of $\ell$-power-free ideals with all prime factors dividing $\ell\O_K$ and $J$ is the set of $\ell$-power-free, prime-to-$\ell$ ideals. Let $A' = \{(u,\mc R,\mf Q,\mf I)\in A : \mc R^\ell [\mf Q\mf I] \text{ is principal}\}.$ Then there exists an $(\ell-1)$-to-1 map
$$\Phi: A' \rightarrow \{L/K : \Gal(L/K) \cong \Z/\ell\Z\}$$
 that sends $(u,\mc R, \mf Q, \mf I)$ to the extension $K(\sqrt[\ell]{\gamma})$ where $\gamma = F(u, \mf R^\ell \mf Q\mf I)$ and $\mf R$ is a fixed prime-to-$\ell$ integral representative of the ideal class $\mc R$. In particular,
$$\#\{L/K : \Gal(L/K) \cong \Z/\ell\Z,\ N(\mf F_{L/K}) \leq X^{\ell-1}\} = \sum_{u \in U/U^\ell} \sum_{\mc R \in \Cl(K)} \sum_{\mf Q \in Q} \sum_{\substack{\mf I \in J\\ \mf I \in (\mf R^\ell[\mf Q])^{-1}\\ N\left(\displaystyle\prod_{\mf p_i \mid \mf I} \mf p_i\right) \leq X}} \frac 1 {\ell-1}.$$
\end{cor}
\begin{proof}
First, we prove that $\Phi$ is surjective. By Kummer theory, every cyclic $\ell$-extension of $L/K$ is given by $K(\sqrt[\ell]\gamma)$ for some $\gamma \in \O_K$. With respect to the fixed set of generators of $P$ chosen by our choice of $F$, we can write $\gamma = \alpha_{\gamma_{\O_K}} u$. By unique factorization of ideals, $\gamma\O_K = \mf R^{\ell} \mf Q\mf I$ where $\mf R^{\ell}$ is the largest $\ell$-power dividing $\gamma$, $\mf Q$ is the $\ell$-factor of $\gamma$, and $\mf I$ is the prime-to-$\ell$, $\ell$-power-free part of $\gamma$. Then $\Phi(i,[\mf R],\mf Q, \mf I) = L$. The map $\Phi$ is $(\ell-1)$-to-1 because if $F(u_1,\mf R_1^\ell Q_1\mf I_1) = F(u_2,\mf R_2^\ell Q_2\mf I_2)^m$ for $1\leq m \leq \ell-1$ then $\Phi(u_1,\mc R_1,Q_1,\mf I_1) = \Phi(u_2,\mc R_2,Q_2,\mf I_2).$ Finally, last equality in the corollary follows from the definition of the map $\Phi$ and the computation of $\mf F_{L/K}$ in Theorem \ref{Daberkow General}.
\end{proof}

\begin{rmk}
    This corollary shows that an $\ell$-Kummer extension is determined by four parameters:
    \begin{enumerate}
        \item A unit $u \in U/U^\ell$
        \item An ideal class $\mc R\in \Cl(K)$
        \item A $\mf Q \in Q$
        \item An $\ell$-power-free, prime-to-$\ell$ ideal $\mf I \in \mc R^{-\ell}[\mf Q]^{-1}$
    \end{enumerate}
    and that the correspondence between these 4-tuples and $\ell$-Kummer extensions is $\ell-1$ to 1.
\end{rmk}

\subsection{Distribution of $\ell$-factors of the discriminant}
\label{LDist}
Theorem \ref{Daberkow General} completely determines all possible $\ell$-factors of the discriminant of an $\ell$-Kummer extension of a number field $K$. Because we plan to sum over each possible $\ell$-factor of the discriminant as in \eqref{SumLPart}, we must first understand density, $\rho_{\mf B}$, of extensions with a given $\ell$-factor of the discriminant, $\mf B$, as in \eqref{rho_q}.

We will use $\rho_{\mf B}$ as an adjustment factor in our enumeration. For a fixed $\mf B$, we will count all $\ell$-Kummer extensions with $N(\mf F_{L/K})\leq \frac{X^{\ell-1}}{N(\mf B)}$ and then multiply this number by $\rho_{\mf B}$ to count the number of $\ell$-Kummer extensions with $N(\Delta_{L/K}) \leq X^{\ell-1}$ and $\mf L_{L/K} = \mf B$. First, we must show that $\rho_{\mf B}$ is a constant. We must also show that it is independent of Steinitz class as described by Equation \eqref{rho_q indep}. 

We saw in Theorem \ref{Daberkow General} that for a prime divisor $\mf q$ of $\ell\O_K$ and an $\ell$-Kummer extension $L = K(\sqrt[\ell]{\gamma})$, $\nu_{\mf q}(\Delta_{L/K})$ is determined by whether $\mf q$ divides $\gamma\O_K$ and the equivalence class of $\gamma$ modulo powers of $\mf q$ if $\mf q$ does not divide $\gamma\O_K$. In other words, for fixed $\mf B$, let
\begin{align*}
W &= \{\mf q | \ell \O_K : \mf q \text{ prime, } \nu_{\mf q}(\mf B) = \ell-1 + \ell\nu_{\mf q}(\ell)\},\\  W' &= \{\mf q | \ell \O_K : \mf q \text{ prime, } \nu_{\mf q}(\mf B) \neq \ell-1 + \ell\nu_{\mf q}(\ell)\}.
\end{align*}\label{W} 
Furthermore, let $$\mf W = \prod_{\mf q \in W'} \mf q^{\ell\nu_{\mf q}(\ell)},$$
 $G = (\O_K/\mf W)^\times$ and $H = G/G^{\ell}$. Also, denote by $\O_K(\mf W)$ the set of elements of $\O_K$ that are relatively prime to $\mf W$ and $\psi: \O_K(\mf W) \rightarrow H$ be the composition of the quotient maps. \label{GH} This notation allows us to state the main results of the section.
 
\begin{prop}
\label{l-Density}
    Let $K$ be a number field with non-trivial class group containing the $\ell$-th roots of unity. Additionally, let $\mf B \in B$, and let $K$, $\rho_{\mf B}$, $H$, $W$, $W'$ and $\mf W$ be as defined above. Let $C$ be the set of congruence classes in $H$ such that for all $\mf q \in W'$, $\gamma \in H$ is an $\ell$-power mod $\mf q^{\ell\nu_{\mf q}(\ell)}$ if $\nu_{\mf q}(\mf B) = 0$ and otherwise $\gamma \in H$ is an $\ell$-power mod $\mf q^s$ and not mod $\mf q^{s+1}$ where $s = -\frac{\nu_{\mf q}(\mf B)}{\ell-1}+\ell\nu_{\mf q}(\ell) + 1$. Finally, let $\mc C \in R_K.$ Then 
   $$\rho_{\mf B}^{\mc C} = \left(\frac{\ell-1}{\ell}\right)^{\#W} \left(\frac 1 \ell\right)^{\#W'}\frac{\#C}{\#H}.$$
\end{prop}

\begin{rmk}
We assume that $K$ has non-trivial class group because the main theorem is only interesting in this case and it allows us to exclude from consideration the field $K=\Q(\sqrt{-3})$ when $\ell=3$. Therefore, for $\ell\geq 3$, we can assume that $[K:\Q] \geq 3$.
\end{rmk}

\begin{cor}\label{SteinitzIndependence}
Let all be as in Proposition \ref{l-Density}. For any $\mc C \in R_K$,
$$\lim_{X\rightarrow \infty} \frac{\#\{L/K: \Gal(L/K) \cong \Z/\ell\Z,\ N(\mf F_{L/K}) \leq X^{\ell-1},\ \mf L_{L/K} = \mf B\}}{\#\{L/K: \Gal(L/K) \cong \Z/\ell\Z,\  N(\mf F_{L/K}) \leq X^{\ell-1}\}} = \rho_{\mf B}^{\mc C}.$$
\end{cor}

By Theorem \ref{Daberkow General}, in order for $\mf L_{L/K} = \mf B$, it must be the case that 
\begin{enumerate}
    \item For all $\mf q \in W$, $\mf q \mid \gamma\O_K$ (divisibility condition),
    \item $\gamma \in C$ (congruence condition).
\end{enumerate}  
Note that congruence classes mod $\mf W$ can be reduced to congruence classes mod $\mf q^{m}$ for all $0 < m \leq \ell\nu_{\mf q}(\ell)$ and all $\mf q \in W'$ by quotienting. The proofs of Proposition \ref{l-Density} and its corollary will be given in Section \ref{3Proofs}. To prove these results, we must understand both the divisibility condition and the congruence condition.

\subsection{Congruence Condition}
We focus on the congruence condition first. In particular, we will show that for any ideal $\mf W$, the set of $\gamma$ relatively prime to $\mf W$ corresponding to $\ell$-Kummer extensions are equidistributed mod $H$.

This equidistribution statement is reminiscent of the standard result on the equidistribution of ideals among ideal classes, often proved using geometry of numbers \cite[\S 6]{M} by counting principal ideals contained in a given ideal. We will adopt that argument with a couple of small changes to count $\gamma \in \O_K$ in a given congruence class modulo $\mf W$. 

Throughout this subsection, let $K$ be a number field of degree $n$ and let $\mf W$ and $\mf I$ be relatively prime ideals in $\O_K$. We will take $\mf W$ to be a product of powers of primes dividing $\ell\O_K$, and $\mf I$ to be the $\ell$-factor and $\ell$-power part of $\gamma\O_K$. Also, let $T = \O_K/\{0\}$ be a multiplicative semi-group with sub-semi-group $V^\ell$, where $V$ is the free part of $U$. Let $Y$ be a set of representatives for cosets of $T/V^\ell$. This is a more formal way of describing the image of $F$ restricted to $P \times V/V^\ell$. \label{V}

We want to count cosets up to $\ell$-powers of units because each coset will correspond to a different $\ell$-Kummer extension. Our goal is to embed $T$ into $(\R^\times)^{r_1} \times (\C^\times)^{r_2}$ and then find a fundamental domain $D$ in which each coset has one representative. Once we have a fundamental domain, we will argue that $Y$ is equally distributed among congruence classes mod $\mf W$ up to $\ell$-powers.

\label{M} We can find a fundamental domain $D$ for $Y$ that is almost the same as the fundamental domain in \cite[p. 115]{M}. Let $M: \O_K\wo \{0\} \rightarrow (\R^\times)^{r_1} \times (\C^\times)^{r_2}$ be the Minkowski embedding $$M(\alpha) = (\sigma_1(\alpha),\ldots, \sigma_{r_1}(\alpha), \tau_1(\alpha), \ldots \tau_{r_2}(\alpha)),$$
let
$\log: (\R^\times)^{r_1} \times (\C^\times)^{r_2} \rightarrow \R^{r_1+r_2}$ be the standard isomorphism
$$\log(x_1,\ldots, x_{r_1}, z_1, \ldots, z_{r_2}) = (\log|x_1|,\ldots, \log|x_{r_1}|, \log|z_1|,\ldots \log |z_{r_2}|),$$
and let $\mu:\O_K\wo \{0\} \rightarrow \R^{r_1+r_2}$ be the composition of the two maps. Also define the norm 
$$N(x_1,\ldots, x_{r_1},z_1,\ldots, z_{r_2}) = x_1\cdots x_{r_1}|z_1|^2\cdots |z_{r_2}|^2$$
on $(\R^\times)^{r_1} \times (\C^\times)^{r_2}$.
We note that $\mu(U)$ is a lattice $\wedge_U$ contained in a hyperplane $H$ of $\R^{r_1+r_2}$. The kernel of $\mu|_U$ is $U_{\tors}$, which is to say that $\mu(V) \cong \wedge_U$. Restricting further,
$\mu(V^\ell)$ is a sub-lattice $\wedge_{V^\ell} \seq \wedge_V = \wedge_U$. 

Let $G$ be the fundamental parallelotope for the lattice $\wedge_{V^\ell}$. Just as in \cite[p. 115]{M}, let $v = (1,\ldots,1,2,\ldots, 2)$ and 
$$D = \{x \in (\R^\times)^{r_1} \times (\C^\times)^{r_2} : \log x \in G + \R v\}.$$
Then \label{D} $D$ is a fundamental domain for $M(B)$. Letting $D_X= \{x\in D : N(x) \leq X\}$, we note that our $D_X$ is $(n-1)$-Lipschitz parametrizable by the same argument as in \cite[p. 118]{M}. The only difference between the fundamental domain of $\wedge_V$ in \cite[p. 115]{M} and our fundamental domain $D$ of $\wedge_{V^\ell}$ is that the inverse image of our $D$ under the log map is $\ell$ times larger in each dimension. Therefore, by the same argument as \cite[p. 118]{M}, we have a $(n-1)$-Lipschitz parametrization of $D_X$ with Lipschitz bound $\lambda = \max(\lambda, \lambda^\ell)$ where $\lambda$ is the constant from \cite[p. 119]{M}.

Now that we understand the fundamental domain, we can argue the equidistribution results that we need.

\begin{prop}
\label{Lattice Dist}
    Let all be as above and $a \in H$. Then    
    $$\lim_{X\rightarrow \infty} \frac{\#\{r \in D_X \cap M(\mf I\cap \O_K(\mf W)) :  \psi(M^{-1}(r)) = a\}}{\#D_X \cap M(\mf I\cap \O_K(\mf W))} = \frac 1 {\#H}.$$
\end{prop}

Before we prove the proposition, we note that the numerator is independent of choice of coset representatives. This is because different coset representatives of $S/V^\ell$ will have a ratio that is an $\ell$-power in $\O_K$ and thus a ratio that is an $\ell$-power modulo $\mf W$.

\begin{proof}   
This result mostly follows from results in \cite[\S 6]{M} since our fundamental domain is still $(n-1)$-Lipschitz parametrizable. The other difference is that we are looking at translations of the lattice $D_X \cap \wedge_{\mf I\mf W}$ by elements of $G$ instead of the original lattice. The proposition follows immediately from the following lemma.
\end{proof}

\begin{lem}
    \label{Translation}
    Let all be as as above and $X \in \R^+$. For notational simplicity, for $c \in \O_K$ relatively prime to $\mf W$, let
    \begin{align*}
        D(\mf I, c) &= \{r \in D\cap M(\mf I \cap \O_K(\mf W)) : \psi(M^{-1}(r)) = c\},\\ D_X(\mf I, c) &= D_X \cap D(c). 
    \end{align*}Then for any $a,b \in \O_K$ relatively prime to $\mf W$, we have
   $$\#D_X(\mf I, a) = \#D_X(\mf I, b) + O(X^{1-\frac 1 n}).$$
\end{lem}

\begin{proof}
Let $F$ denote the fundamental parallelepiped of the lattice corresponding to $D \cap \mf I\mf W$ under the logarithmic embedding. Then $$\log(D(\mf I, c)) = \bigcup_{\substack{d \in F\\
\phi(d) = c}} (\log(D \cap M(\mf I \mf W)) + d)$$
where sets in the union are translations of the lattice corresponding to $D_X \cap M(\mf I \mf W)$ by $d$. When considering the lattices bounded by $X$, the only possible difference in the number of points in $D_X(\mf I, a)$ and $D_X(\mf I, b)$ comes from the cells of the lattice that intersect the curve that gives the norm condition. As shown in the proof of \cite[Lemma 6.2]{M}, the number of these cells grows at $O(X^{1-\frac 1 n})$ while the main term grows at $O(X)$. 
\end{proof}

We now give more explicit constants for $D_X(\mf I,c)$. The following result is well-known and is the content of \cite[\S 6, Exercise 13]{M}.
\begin{prop}
\label{Equidistribution of Ideals}
Let $K$ be a number field of degree $n$ and $\mc C \in \Cl(K)$. Also, let $\mf W$ be an ideal in $K$ and
$$i_{\mc C, \mf W}(X) = \{\mf A \seq \O_K : N(\mf A) \leq X,\ \gcd(\mf A, \mf W) = \O_K,\ \mf A \in \mc C\}$$
Then 
$$\#i_{\mc C, \mf W}(X) = \frac{\displaystyle\res_{s=1} \zeta_K(s)}{h_K}\prod_{\substack{\mf q \mid \mf W\\ \mf q \text{ prime}}}\left(1-\frac{1}{N(\mf q)}\right)X + O(X^{1-\frac 1 n}).$$
\end{prop}

From Proposition \ref{Equidistribution of Ideals}, we can conclude that for given ideals $\mf W$ and $\mf I$
$$\#(i_{\mc C,\mf W}(X) \cap \{\mf A \seq \O_K : \mf I \mid \mf A\}) = \frac{\displaystyle\res_{s=1} \zeta_K(s)}{h_K N(\mf I)}\prod_{\substack{\mf q \mid \mf W\\ \mf q \text{ prime}}}\left(1-\frac{1}{N(\mf q)}\right)X + O\left(\left(\frac{X}{N(\mf I)}\right)^{1-\frac 1 n}\right).$$

Switching from counting ideals using the lattice $\wedge_{V}$ to counting coset representatives using the lattice $\wedge_{V^\ell}$ requires multiplication by $\ell^{r_1+r_2}$. We also drop the ideal class condition. Let $D_{X}(\mf I) = D_X \cap M(\mf I \cap \O_K(\mf W))$, then
$$\#D_{X}(\mf I) = \frac{\ell^{r_1+r_2}\displaystyle\res_{s=1} \zeta_K(s)}{N(\mf I)}\prod_{\substack{\mf q \mid \mf W\\ \mf q \text{ prime}}}\left(1-\frac{1}{N(\mf q)}\right)X + O\left(\left(\frac{X}{N(\mf I)}\right)^{1-\frac 1 n}\right).$$

Using Proposition \ref{Lattice Dist}, we can improve this further to compute that for any $c \in H$,
$$D_X(\mf I, c) = \frac{\ell^{r_1+r_2}\displaystyle\res_{s=1} \zeta_K(s)}{N(\mf I)\#H}\prod_{\substack{\mf q \mid \mf W\\ \mf q \text{ prime}}}\left(1-\frac{1}{N(\mf q)}\right)X + O\left(\left(\frac{X}{N(\mf I)}\right)^{1-\frac 1 n}\right).$$

Looking ahead to our intended application in counting $\ell$-Kummer extensions, we would like to restrict our set of representatives to those that are $\ell$-power-free. We argue this by inclusion-exclusion.

\begin{prop}
\label{Lattices}
    Let all notation be as above, and $m\geq 2$ be a positive integer and $$\mf N = \bigcup_{(\mf P,\mf I\mf W) = \O_K} \mf P^m.$$ Also, let
    \begin{align*}
        D_X'(\mf I) &= D_X \cap M(\mf I \cap \O_K(\mf W)) \wo M(\mf N)\\
    D_X'(\mf I,a) &= D'_X(\mf I) \cap D_X(\mf I, a).
    \end{align*}
    Then 
    $$\lim_{X\rightarrow \infty} \frac{\#D'_X(\mf I,a)}{\#D'_X(\mf I)} =\frac 1 {\#H}.$$
\end{prop}

\begin{proof}
By inclusion-exclusion, 
\begin{align*} D_X'(\mf I) 
&= \sum_{\substack{(\mf A, \mf I) = \O_K\\ (\mf A, \mf W) = \O_K, \\ N(\mf I\mf A^m) \leq X}} \mu(\mf A) D_X(\mf I\mf A^m)\\
&=  \sum_{\substack{(\mf A, \mf I) = \O_K\\ (\mf A, \mf W) = \O_K, \\ N(\mf I\mf A^m) \leq X}}\mu(\mf A)\left( \frac{\ell^{r_1+r_2}\displaystyle\res_{s=1} \zeta_K(s)}{N(\mf I\mf A^m)}\left(\prod_{\substack{\mf q \mid \mf W\\ \mf q \text{ prime}}}1-\frac{1}{N(\mf q)}\right)X + O\left(\left(\frac{X}{N(\mf I\mf A^m)}\right)^{1-\frac 1 n}\right)\right)\\
&= \frac{\ell^{r_1+r_2}\displaystyle\res_{s=1} \zeta_K(s)}{N(\mf I\mf )}\left(\prod_{\substack{\mf q \mid \mf W\\ \mf q \text{ prime}}}1-\frac{1}{N(\mf q)}\right)\left(\sum_{\substack{(\mf A, \mf I) = \O_K\\ (\mf A, \mf W) = \O_K, \\ N(\mf I\mf A^m) \leq X}}\frac{\mu(\mf A)}{N(\mf A^m)}\right)X  \\
&+ O\left(\left(\frac{X}{N(\mf I)}\right)^{1-\frac 1 n}\sum_{\substack{(\mf A, \mf I) = \O_K\\ (\mf A, \mf W) = \O_K, \\ N(\mf I\mf A^m) \leq X}} \frac{1}{N(\mf A^{m-\frac m n})}\right)\\
&= \frac{\sigma_{\mf I\mf W}D_X(\mf I)}{\zeta_K(m)} + O\left(\left(\frac{X}{N(\mf I)}\right)^{1-\frac 1 n}\log\left(\frac{X}{N(\mf I)}\right)\right)
\end{align*}
where $\sigma_{\mf I\mf W}$ is a constant coming from the ideals that are not relatively prime to $\mf I$ or $\mf W$. Note that
$$\sum_{\substack{(\mf A, \mf I) = \O_K\\ (\mf A, \mf W) = \O_K, \\ N(\mf I\mf A^m) \leq X}} \frac{1}{N(\mf A^{m-\frac m n})} = O\left(\log\left(\frac{X}{N(\mf I)}\right)\right)$$
for $m=n=2$ and $O(1)$ otherwise because
$$\sum_{N(\mf A) \leq X} \frac 1 {N(\mf A)} = \sum_{j \leq X} \frac{\#\{\mf A \mid N(\mf A) = j\}}{j} = O(\log X)$$
and $\zeta_K(s)$ converges for $\Re(s)> 1$.

By the same logic,
$$D_X'(\mf I,a) = \frac{\sigma_{\mf I\mf W}D_X(\mf I,a)}{\zeta_K(m)} + O\left(\left(\frac{X}{N(\mf I)}\right)^{1-\frac 1 n}\log\left(\frac{X}{N(\mf I)}\right)\right).$$
Therefore,
    $$\lim_{X\rightarrow\infty} \frac{D'_X(\mf I,a)}{D_X'(\mf I)} = \frac 1 {\#H}.$$
\end{proof}

By design, our choices of $r \in D_X(\mf I,c)$ correspond to representatives $\gamma$ of equivalence classes of $\ell$-Kummer extensions with a given $\ell$-factor and $\ell$-power part (combined in the ideal $\mf I$). Morally speaking, this means that the likelihood of a representative $\gamma$ of an equivalence class of $\ell$-Kummer extensions is equally likely to be in any congruence class (up to $\ell$-powers) modulo $\mf W$. This means that we can compute the density of $\gamma$ in such equivalence classes simply by computing the structure of the group $H$. 

\begin{rmk}
\label{Utors}
    We should note that we also have to worry about the torsion subgroup of $U_{\tors} \seq U$. Because we have not quotiented by $U_{\tors}^\ell$ in our definition of $Y$, we have counted $\frac{\#U_{\tors}}{\ell}$ $\gamma \in \O_K$ corresponding to each $\ell$-Kummer extension. However, this is not an issue because $U_{\tors}$ is a finite group, each element of a coset $U_{\tors}/U_{\tors}^\ell$ will be equivalent up to $\ell$-powers mod $\mf W$ by definition, and each coset of $U_{\tors}^\ell$ has the same number of elements. Therefore, we can safely ignore this complication when moving from our lattice result to the corresponding result for $\ell$-Kummer extensions.
\end{rmk}

\subsection{Divisibility Condition}
In this subsection, we count prime-to-$\ell$, $\ell$-power-free ideals $\mf I$. Due to the fact that the resulting field discriminant is dependent on not $N(\mf I)$, but $N(\rad(\mf I))$, we need a result for equidistribution sorted by this norm. The proofs of next proposition and its following corollaries use the Selberg-Delange method and are based on the proofs of Gaudet-Wong in \cite[\S 9]{GW}.

\begin{prop}
\label{Divisors}
    Let $k$ be a positive integer and $K$ be a number field. Then the number of $k$-tuples of squarefree, pairwise coprime ideals $(\mf I_1,\ldots,\mf I_k)$ relatively prime to $\ell\O_K$ with $N\left(\prod_{i=1}^k \mf I_i\right) \leq X$ is equal to
    $$\lambda X \log^{k-1} X + O(X\log^{k-2} X)$$
    where
    $$\lambda = \frac{\left(\displaystyle\res_{s=1}\zeta_K(s)\right)^k}{(k-1)!} \left[\prod_{(\mf p,\ell\O_K)=\O_K} \left(1+\frac k {N(\mf p)}\right)\left(1-\frac 1 {N(\mf p)}\right)^k\right]\left[\prod_{\mf q \mid \ell\O_K} \left(1-\frac{1}{N(\mf q)}\right)^k\right].$$
\end{prop}

\begin{proof}
    Instead of counting $k$-tuples, we can instead count the number of ways of dividing each squarefree ideal with norm less than $X$ into $k$ factors. These factors are the elements of the $k$-tuples. That is to say, we must compute the sum
    $$\sum_{\substack{(\mf I,\ell\O_K) = \O_K\\N(\mf I) \leq X}} d_k(\mf I) \mu^2(\mf I)$$
    where $d_k(\mf I)$ is the number of ways that $\mf I$ can be written as the product of $k$ ideals, including different orderings. Note that if $\mf I$ is squarefree, then its factors must be squarefree and relatively prime.
    We intend to approximate this sum using a Tauberian theorem, so study the corresponding Dirichlet series
    $$F(s) =\sum_{(\mf I,\ell\O_K) = \O_K} \frac{d_k(\mf I) \mu^2(\mf I)}{N(\mf I)^s}$$
    which converges absolutely and uniformly for $\Re(s) > 1$. On this half-plane of convergence, $F(s)$ has an Euler product
    $$\prod_{(\mf p,\ell\O_K) = \O_K} \left(1 + \frac{k}{N(\mf p)^{s}}\right).$$
    Note that
    $$F(s) = G(s)\zeta_K(s)^{k}\prod_{\mf q \mid \ell\O_K} \left(1-\frac 1 {N(\mf q)^s}\right)^k$$
    where
    $$G(s) = \prod_{\mf p} \left(1+\frac k {N(\mf p)^s}\right)\left(1-\frac 1 {N(\mf p)^s}\right)^k$$
    is holomorphic for $\Re(s) > \frac 1 2$.
    By the Wiener-Ikehara Tauberian theorem (see \cite[Appendix II]{N},  \cite[\S II.5]{T}, or \cite{Murty}),
    $$
    \sum_{\substack{(\mf I,\ell\O_K) = \O_K\\N(\mf I) \leq X}} d_k(\mf I) \mu^2(\mf I) = 
    \lambda_k X \log^{k-1}(X) + O(X\log^{k-2} X) $$
    where \begin{align*}
    \lambda_k
    &= \frac {\left(\displaystyle\res_{s=1}\zeta_K(s)\right)^k} {\Gamma(k)}G(1) \prod_{\mf q \mid \ell\O_K} \left(1-\frac 1 {N(\mf q)}\right)^k  \\
    &= \frac{\left(\displaystyle\res_{s=1}\zeta_K(s)\right)^k}{(k-1)!} \left[\prod_{(\mf p,\ell\O_K)=\O_K} \left(1+\frac k {N(\mf p)}\right)\left(1-\frac 1 {N(\mf p)}\right)^k\right]\left[\prod_{\mf q \mid \ell\O_K} \left(1-\frac{1}{N(\mf q)}\right)^k\right].  
    \end{align*}
\end{proof}

Note that $\prod_{i=1}^{\ell-1} \mf I_i^i$ will be the prime-to-$\ell$, $\ell$-power-free part of some principal ideal $\gamma\O_K$. Therefore, we must make sure that it is in the correct ideal class.

\begin{cor}
\label{Equidistribution of Radicals}
    Let $k$ be a positive integer, $K$ be a number field with class number $h_K$. Also let $\lambda_k$ be the constant defined above and let $\mc C \in \Cl(K)$. Then the number of $k$-tuples of squarefree, prime-to-$\ell$ and pairwise coprime ideals $(\mf I_1,\ldots,\mf I_k)$ relatively prime to $\ell\O_K$ with $N\left(\prod_{i=1}^k \mf I_i\right) \leq X$ and $\prod_{i=1}^k \mf I_i^i \in \mc C$ is given by
    $$\frac {\lambda_k}{h_K} X \log^{k-1} X + O(X\log^{k-2} X).$$
\end{cor}

\begin{proof}
    For $1 \leq i \leq k$ let $\chi_i \in \widehat{\Cl(K)}$ and define $$\rho_{\chi_1,\ldots,\chi_k}(\mf I) = \sum_{\mf I = \prod_{i=1}^k \mf I_i} \prod_{i=1}^k \chi_i(\mf I_i).$$
    If all characters are trivial, then $\rho_{\chi_1,\ldots,\chi_k}(\mf I) = d_k(\mf I)$. We also define
    $$F_{\chi_1,\ldots,\chi_{k}}(s) = \sum_{(\mf I,\ell\O_K) = \O_K} \frac{\rho_{\chi_1,\ldots,\chi_k}(\mf I)\mu^2(\mf I)}{N(\mf I)^s}.$$
    As in the proof of Proposition \ref{Divisors}, we have the factorization
    $$F_{\chi_1,\ldots,\chi_{k}}(s) = \prod_{(\mf p,\ell\O_K) = \O_K} \left(1+\frac{\sum_{i=1}^k \chi_i(\mf I)}{N(\mf p^s)}\right) = H(s)\prod_{i=1}^{\ell-1}L_K(s,\chi_i)$$
    where $L_K(s,\chi_i)$ is the Hecke L-function associated to $\chi_i$ and $H(s)$ is holomorphic for $\Re(s) > \frac 1 2$. 
    By Proposition \ref{Divisors}, if all the characters are trivial, then 
    $$\sum_{\substack{(\mf I,\ell\O_K) = \O_K\\N(\mf I) \leq X}} \mu^2(\mf I)\rho(\mf I) = \lambda_k X \log^{k-1} X + O(X\log^{k-2} X).$$
    Otherwise, the right-hand side has at most $k-1$ poles at $s=1$ and an application of Perron's formula gives
    $$\sum_{\substack{(\mf I,\ell\O_K) = \O_K\\N(\mf I) \leq X}} \rho(\mf I)\mu^2(\mf I)= O(X \log^{k-2} X).$$ 

    The below lemma shows that $h_K^{k-1}$ of the $h_K^{k}$ possible $k$-tuples of ideal classes $(\mc C_1,\ldots,\mc C_k) \in \prod_{i=1}^k\Cl(K)$ have the property that
    $$\prod_{i=1}^k \mc C_i^i = \mc C.$$
    Using this fact and the Hecke characters to detect ideal class, we conclude that the number of $k$-tuples of squarefree and pairwise coprime ideals $(\mf I_1,\ldots,\mf I_k)$ with $N\left(\prod_{i=1}^k \mf I_i\right) \leq X$ and $\prod_{i=1}^k \mf I_i^i \in \mc C$ is given by
    $$\frac{\lambda_k}{h_K} X \log^{k-1} X + O(X\log^{k-2} X).$$
\end{proof}

For lack of a convenient reference, we now state and prove the purely group theoretic result we needed for the previous proof.

 \begin{lem}
    \label{Group product distribution}
    Let $G$ be a finite abelian group and fix a $g \in G$. Then for $m \geq 1$
    $$\#\{(g_1,g_2,\ldots,g_{m}) \in G^m : \prod_{i=1}^{m} g_i^{i} = g\} = \#G^{m-1}.$$
    In other words, these products are equally distributed between elements of $G$.
 \end{lem}

\begin{proof}
    For $m=1$, the result is trivial, so we assume $m > 1$. For notational simplicity, let $G^m = \prod_{i=1}^m G$. Also, let $f: G^m \rightarrow G$ be the homomorphism $f(g_1,\ldots, g_m) = \prod_{i=1}^{m} g_i^{i}$. This homomorphism is surjective because for $g \in G$, $f(g^{-1},g,1,\ldots, 1) = g$. The cosets of the kernel are the sets $$\#\{(g_1,g_2,\ldots,g_m) \in G^m : \prod_{i=1}^{m} g_i^{i+j} = g\} = \#G^{m-1}$$ for each $g \in G$. Each coset has size $\#G^{m-1}$.
\end{proof}

We need one refinement to this result. In the case that we want to fix a Steinitz class, we will require the $\ell$-power free part of the ideal $\gamma$ to satisfy two class group conditions, one to ensure that $K(\sqrt[\ell]{\gamma})$ has the correct Steinitz class, and one to ensure $\gamma\O_K$ is principal. More explicitly, we need the following corollary.

\begin{cor}
\label{Equidistribution of Radicals 2}
    Let $k > 1$ be an integer, $K$ be a number field with class number $h_K$. Also let $\lambda_k$ be the constant defined above and let $\mc C_1,\mc C_2 \in \Cl(K)$. Then the number of $k$-tuples of squarefree and pairwise coprime ideals $(\mf I_1,\ldots,\mf I_k)$ relatively prime to $\ell\O_K$ with $N\left(\prod_{i=1}^k \mf I_i\right) \leq X$, $\prod_{i=1}^k \mf I_i^i \in \mc C_1$, and $\prod_{i=2}^k \mf I_i^{i-1} \in \mc C_2$ is given by
    $$\frac {\lambda_k}{h_K^2} X \log^{k-1} X + O(X\log^{k-2} X).$$
\end{cor}

This corollary follows from the same method as above and a revised group theory lemma.
\begin{lem}
    \label{Group product distribution 2}
    Let $G$ be a finite abelian group and fix some $a,b \in G$. Then for $m \geq 2$,
    $$\#\{(g_1,g_2,\ldots,g_{m}) \in G^m : \prod_{i=1}^{m} g_i^{i} = a \text{ and } \prod_{i=2}^m g_i^{i-1} = b\} = \#G^{m-2}.$$
 \end{lem}

\begin{proof}
    By Lemma \ref{Group product distribution},
    $$\#\{(g_1,g_2,\ldots,g_{m}) \in G^m : \prod_{i=2}^m g_i^{i-1} = b\} = \#G^{m-1}.$$
    Note also that for every element $(g_1,\ldots, g_m)$ in this set, the element $(g,g_2,\ldots, g_m)$ is in this set for all $g \in G$. For every tail $(g_2,\ldots, g_m)$ there is exactly one $g_1$ that satisfies the $\prod_{i=1}^m g_i^i = a$ condition, namely $g_1 = a \left(\prod_{i=2}^m g_i^i\right)^{-1}$. Therefore,
    $$\#\{(g_1,g_2,\ldots,g_{m}) \in G^m : \prod_{i=1}^{m} g_i^{i} = a \text{ and } \prod_{i=2}^m g_i^{i-1} = b\} = \frac{\#G^{m-1}}{\#G} = \#G^{m-2}.$$
    \end{proof}

\subsection{Existence and independence of $\rho_{\mf B}$}
\label{3Proofs}
Now that we have covered both the divisibility condition and the congruence condition, we can state and prove the main results of the section.

\begin{proof}[Proof of Proposition \ref{l-Density}]
    In order for $\mf L_{L/K} = \mf B$, $\nu_{\mf q}(\Delta_{L/K}) = \nu_{\mf q}(\mf B)$ for all $\mf q \in W$ and all $\mf q \in W'$. We start with the $\mf q \in W$, which is to say $\mf q \mid \gamma\O_K$. We can assume that the $\ell$-factor of $\gamma$ is $\ell$-power-free using Lemma \ref{Reduction General}. Thus, we need $\nu_{\mf q}(\gamma\O_K) \in \{1,\ldots,\ell-1\}$. We now consider the rest of the ideal factorization of $\gamma\O_K$. We know that $\gamma\O_K$ must have the rest of the $\ell$-factor relatively prime to $\mf q$ which we will call $\mf Q'$, an ideal class $\mc R$ for the $\ell$-power-part $\mf R^{\ell}$, and an $\ell$-power free part $\mf I$. 

    Fix some $\mf Q'$ and $\mc R$. Given that $\gamma\O_K$ is principal we must restrict our choices of $\mf I$ to prime-to-$\ell$, $\ell$-power-free ideals $\mf I$ such that 
    $$\prod_{i=1}^{\ell-1} \mf I_i^{i} \in \left[\mf q^m \mf Q'\mf R^{\ell} \right]^{-1}$$ 
    so $\mf q^m \mf Q'\mf R^{\ell}\mf I$ is principal. We call this the ``principality condition" on $\mf I$.
    By Proposition \ref{Steinitz Formula Prop}, since we want $\St(\O_L) = \mc C$, we also require
    $$\left(\prod_{i=2}^{\ell-1} \mf I_i^{i-1}\right)^{\frac{\ell-1}{2}} \in \mc C^{-1}\left[\sqrt{\frac{\mf B}{(\mf R^{\ell}\mf q^m\mf Q)^{\ell-1}}}\ \right]$$ 
    where $\mf I_i$ is the $i$-th-power-part of $\mf I$. We call this the ``Steinitz condition" on $\mf I$. Note that this condition requires $\mc C \in R_K$. For simplicity, we denote by $J_{\mf B,\mf q^m\mf Q', \mc R, \mc C}(X)$ the set of such $\mf I$ with $N(\rad(\mf I)) \leq X$. This norm condition ensures that $N(\mf F_{L/K}) \leq X^{\ell-1}$ by Theorem \ref{Daberkow General}. We enumerate $J_{\mf B, \mf q^m\mf Q', \mc R, \mc C}\left(X\right)$ by counting $\ell-1$ tuples $(\mf I_1,\ldots, \mf I_{\ell-1})$ of relatively prime squarefree ideals with norm of their product less than $X$ that satisfy the conditions above. By Corollary \ref{Equidistribution of Radicals 2}, and the fact that cosets of $\Cl(K)/R_K$ are the same size,
    $J_{\mf B, \mf q^m\mf Q', \mc R, \mc C}\left(X\right)$ is asymptotically equivalent for all choices of $\mf B$, $m$, $\mf Q'$, $\mc R$ and also for all choices of the ideal class $\mc C$.
    
    For $m \neq 0$ and any $u \in U/U^{\ell}$, $\Phi(u,\mc R, \mf q^m\mf Q',\mf I)$ gives extensions with $\nu_{\mf q}(\Delta_{L/K}) = \nu_{\mf q}(\mf B)$. For $m = 0$ and any $u \in U/U^{\ell}$, $\Phi(u,\mc R, \mf q^m\mf Q',\mf I)$ gives extensions with $\nu_{\mf q}(\Delta_{L/K}) \neq \nu_{\mf q}(\mf B)$. Summing over the finitely many choices of $\mf Q'$ and $\mc R$ gives equidistribution among powers of $\mf q$ for ideal representatives of each equivalence class of $\ell$-Kummer extensions. Therefore, for each $\mf q \in W$, we get a factor of $\frac{\ell-1}{\ell}$ in $\rho_{\mf B}^{\mc C}$ because there are $\ell$ equally distributed choices for $\nu_{\mf q}(\gamma\O_K)$ and $\ell-1$ of them give the correct $\nu_{\mf q}(\mf L_{L/K})$.

    Next, we handle the $\mf q \in W'$. As noted above, $\frac 1 {\ell}$ of the choices for $\nu_{\mf q}(\gamma\O_K)$ can give the correct $\nu_{\mf q}(\mf L_{L/K})$. This gives us a factor of $\frac 1 \ell$ in $\rho_{\mf B}^{\mc C}$ for each $\mf q \in W'$. However, we also need $\gamma$ to satisfy a congruence condition mod $H$ to have the correct $\nu_{\mf q}(\mf L_{L/K})$ for each $\mf q \in W'$.
    
    First, assume $\ell > 2$. We fix a product of powers of $\mf q' \in W$, which we call $\mf Q \in Q'$, to be the $\ell$-factor of $\gamma\O_K$. We also fix an ideal class $\mc R \in \Cl(K)$ and representative $\mf R \in \mc R$ such that $\mf R^\ell$ will be the $\ell$-power-part of $\gamma\O_K$. Next, we fix an $\ell$-power-free, prime-to-$\ell$ squarefull ideal $\mf I' = \prod_{i=2}^\ell \mf I_i$ such that
  $$\left(\prod_{i=2}^{\ell-1} \mf I_i^{i-1}\right)^{\frac{\ell-1}{2}} \in \mc C\left[\sqrt{\frac{\mf B}{(\mf R^{\ell}\mf Q)^{\ell-1}}}\ \right].$$ 
    Note that these ideals $\mf I'$ satisfy the Steinitz condition, but not a principality condition. We use $J'_{\mf B,\mf Q, \mc R, \mc C}(X)$ to denote such $\mf I'$ with $N(\rad(\mf I')) \leq X$. We can count $J'_{\mf B,\mf Q, \mc R, \mc C}(X)$ using Corollary \ref{Equidistribution of Radicals} with minor re-indexing.
    Finally, we use $\mf Q\mf R^{\ell}\mf I'$ as the ideal $\mf I$ in Proposition \ref{Lattices}. By this Proposition, representatives $\gamma$ that are multiples of $\mf Q \mf R^{\ell}\mf I'$ and relatively prime to $\mf W$ of $\ell$-Kummer extensions are equidistributed in $H$.    

    To be more explicit, given some $w = (w_{\mf q})_{\mf q \in W'}$ with $w_{\mf q} \neq \ell-1+\ell\nu_{\mf q}(\ell)$ a possible value for $\nu_{\mf q}(\mf L_{L/K})$, and $C$ the subset of congruence classes in $H$ that result such that if $\gamma \in C$, $\nu_{\mf q}(\mf L_{K(\sqrt[\ell]\gamma)/K}) = w_{\mf q}$ for all $\mf q \in W'$,
    we are computing the sum
    \begin{align*}
    S_w(X^{\ell-1}) &=\frac{\#\{L/K: \Gal(L/K) \cong \Z/\ell\Z,\ \ N(\mf F_{L/K}) \leq X^{\ell-1},\ \St(\O_L) = \mc C,\ \nu_{\mf q}(\mf L_{L/K}) = w_{\mf q}\}}{\left(\frac 1 \ell\right)^{\#W'}\#(U_{\tors}/U_{\tors}^{\ell})(\ell-1)}\\
    &=\sum_{a \in C}\sum_{\mf Q\in Q'} \sum_{\mc R \in \Cl(K)}\sum_{\mf I' \in J'_{\mf B,\mf Q, \mc R, \mc C}(X)}D_{\frac{ X N(\mf Q\mf R^\ell\mf I')}{N(\rad(\mf I'))}}'(\mf Q\mf R^{\ell}\mf I',a).
    \end{align*}
    We can count the set of these extensions with the nested sums due to Corollary \ref{Choices}.
    We separate out the $\left(\frac 1 \ell\right)^{\#W'}\#(U_{\tors}/U_{\tors}^{\ell})(\ell-1)$ factor for convenience. We already justified why we can ignore $U_{\tors}$ (Remark \ref{Utors}) and must include $\left(\frac 1 \ell\right)^{\#W'}$. The $\ell-1$ factor comes from the fact that $\Phi$ in Corollary \ref{Choices} is an $(\ell-1)$-to-1 map. 
    We use $\frac{ X N(\mf Q\mf R^\ell\mf I')}{N(\rad(\mf I'))}$ as the norm bound on $r \in D'(\mf Q\mf R^{\ell}\mf I',a)$ to ensure that $$ N(\mf L_{L/K}) = N(\rad(\mf I')^{\ell-1}\mf I_1^{\ell-1}) \leq X^{\ell-1}$$ since $N(\mf I_1) = \frac{N(r)}{N(\mf Q\mf R^\ell \mf I')}$.
    By Proposition \ref{Lattices},
    \begin{align*}
    S_w(X^{\ell-1})&=\#C\sum_{\mf Q\in Q'} \sum_{\mc R \in \Cl(K)}\sum_{\mf I' \in J'_{\mf B,\mf Q, \mc R, \mc C}(X)}\left[\frac{\sigma_{\mf I\mf W}\ell^{r_1+r_2}\displaystyle\res_{s=1} \zeta_K(s)}{\zeta_K(2)N(\mf Q\mf R^{\ell}\mf I')}\prod_{\substack{\mf q \mid \mf W\\ \mf q \text{ prime}}}\left(1-\frac{1}{N(\mf q)}\right)\frac{ X N(\mf Q\mf R^\ell\mf I')}{N(\rad(\mf I'))}\right] \\
   &+ O\left(\sum_{\mf Q\in Q'} \sum_{\mc R \in \Cl(K)}\sum_{\mf I' \in J'_{\mf B,\mf Q, \mc R, \mc C}(X)}\left(\frac{\frac{ X N(\mf Q\mf R^\ell\mf I')}{N(\rad(\mf I'))}}{N(\mf Q\mf R^{\ell}\mf I')}\right)^{1-\frac 1 n}\right)\\
   &=\left(\#C \frac{\sigma_{\mf I\mf W}\ell^{r_1+r_2}\displaystyle\res_{s=1} \zeta_K(s)}{\zeta_K(2)}\prod_{\substack{\mf q \mid \mf W\\ \mf q \text{ prime}}}\left(1-\frac{1}{N(\mf q)}\right)X\sum_{\mf Q\in Q'} \sum_{\mc R \in \Cl(K)} \sum_{\mf I' \in J'_{\mf B,\mf Q, \mc R, \mc C}(X)} \frac 1 {N(\rad(\mf I'))}\right) \\
   &+ O\left(X^{1-\frac 1 n} \sum_{\mf I' \in J'_{\mf B,\mf Q, \mc R, \mc C}(X)} \frac{1}{N(\rad(\mf I'))^{1-\frac 1 n}}\right).
   \end{align*}
   As mentioned after the statement of the proposition, we are allowed to use the smaller error term $O(X^{1-\frac 1 n})$ instead of $O(X^{1-\frac 1 n}\log X)$ from Proposition \ref{Lattices} since if $\ell \geq 3$, we can assume that $[K:\Q] \geq 3$.
   By Proposition \ref{Equidistribution of Radicals} with minor re-indexing, and the fact that the cosets of $\Cl(K)/R_K$ are the same size,
   \begin{align*}
   S_w(X^{\ell-1})&=  \frac{\#C\sigma_{\mf I\mf W}\ell^{r_1+r_2}\displaystyle\res_{s=1} \zeta_K(s)}{\zeta_K(2)}\prod_{\substack{\mf q \mid \mf W\\ \mf q \text{ prime}}}\left(1-\frac{1}{N(\mf q)}\right)X\sum_{\mf Q\in Q'} \sum_{\mc R \in \Cl(K)}\left(\frac{\lambda_{\ell-2}}{\#R_K} \log^{\ell-2} X + O\left(\log^{\ell-3} X\right)\right)\\
   &+ O\left(X^{1-\frac 1 n} \left(\frac{X\log^{\ell-3} X}{X^{1-\frac 1 n}}\right)\right).
   \end{align*}
   Note that we have $\frac{\lambda_{\ell-2}}{\#R_K}$ because $\left(\frac{\lambda_{\ell-2}}{h_K}\right)\left(\frac{h_K}{\#R_K}\right) = \frac{\lambda_{\ell-2}}{\#R_K}$ where $\frac{h_K}{\#R_K}$ gives the number of elements $\mc D \in\Cl(K)$ such that $\mc D^{\frac{\ell-1}{2}} = \mc C$ for a fixed $\mc C \in R_K$. Finally, by Abel summation,
   \begin{align*}
    S_w(X^{\ell-1})&=\frac{\#C\sigma_{\mf I\mf W}\ell^{r_1+r_2}\lambda_{\ell-2}\displaystyle\res_{s=1} \zeta_K(s)}{\zeta_K(2)\#R_K}\prod_{\substack{\mf q \mid \mf W\\ \mf q \text{ prime}}}\left(1-\frac{1}{N(\mf q)}\right)\sum_{\mf Q\in Q'} \sum_{\mc R \in \Cl(K)}\left(X\log^{\ell-2} X\right) + O\left(X\log^{\ell-3}X\right).
    \end{align*}
    
    By the same logic, summing over all $a \in H$ instead of $a \in C$ we can compute
    \begin{align*}
    S_{W'}(X^{\ell-1}) &= \frac{\#\{L/K: \Gal(L/K) \cong \Z/\ell\Z,\ \ N(\mf F_{L/K}) \leq X^{\ell-1},\ \St(\O_L) = \mc C, \nu_{\mf q}(\mf L_{L/K}) \neq \ell\nu_{\mf q}(\ell)\}}{\left(\frac 1 \ell\right)^{\#W'}\#(U_{\tors}/U_{\tors}^{\ell})(\ell-1)}\\
    &= \frac{\#H\sigma_{\mf I\mf W}\ell^{r_1+r_2}\lambda_{\ell-2}\displaystyle\res_{s=1} \zeta_K(s)}{\zeta_K(2)\#R_K}\prod_{\substack{\mf q \mid \mf W\\ \mf q \text{ prime}}}\left(1-\frac{1}{N(\mf q)}\right)\sum_{\mf Q\in Q'} \sum_{\mc R \in \Cl(K)}\left(X\log^{\ell-2} X\right) + O\left(X\log^{\ell-3}X\right).
    \end{align*}
    For our purposes, the constant is not important. We just needed to confirm that the error term remains asymptotically smaller than the main term.
    
    For $\ell = 2$, we have the larger error term from Proposition \ref{Lattices}, but we do not need the sum over $J'_{\mf B,\mf Q,\mc R,\mc C}(X)$. Instead, the Steinitz class is determined by $\mc R \in \Cl(K)$ and the fixed choices of $\mf B$ and $\mf Q$ so for all notation as above and a choice of representative $\mf R \in \mc C^{-1} \left[\sqrt{\frac{\mf B}{\mf Q}}\right]$,
    \begin{align*}S_w(X)
    &=\sum_{a \in C}\sum_{\mf Q\in Q'} D_{X N(\mf Q\mf R^2)}(\mf Q\mf R^2\mf,a)\\ 
   &=\#C\sum_{\mf Q\in Q'}\left[\ell^{r_1+r_2}\displaystyle\res_{s=1} \zeta_K(s)\prod_{\substack{\mf q \mid \mf W\\ \mf q \text{ prime}}}\left(1-\frac{1}{N(\mf q)}\right)X\right] + O\left( X^{1-\frac 1 n}\log X\right)
   \end{align*}
   and 
    \begin{align*}S_{W'}(X) =
   \#H\sum_{\mf Q\in Q'}\left[\ell^{r_1+r_2}\displaystyle\res_{s=1} \zeta_K(s)\prod_{\substack{\mf q \mid \mf W\\ \mf q \text{ prime}}}\left(1-\frac{1}{N(\mf q)}\right)X\right] + O\left( X^{1-\frac 1 n}\log X\right).
   \end{align*}
   
    Therefore, for any prime $\ell$, the density of $\gamma \in \O_K$ corresponding to $\ell$-Kummer extensions for which $\nu_{\mf q}(\mf L_{L/K}) = \nu_{\mf q}(\mf B)$ for all $\mf q \in W'$ is determined by the ratio of the size of $C$ (defined using Theorem \ref{Daberkow General}) to the size of $H$. As with the $\mf q \in W$, we can then sum over the finitely many possible $\mf Q$ and $\mc R$, and the result follows.
\end{proof}

\begin{proof}[Proof of Corollary \ref{SteinitzIndependence}]
Note that the computation of $\rho_{\mf B}^{\mc C}$ in the proof of Proposition \ref{l-Density} is independent of $\mc C \in R_K$ since $\#J_{\mf B,\mf Q,\mc R,\mc C}(X)$ and $\#J_{\mf B,\mf Q,\mc R,\mc C}'(X)$ are asymptotically independent of $\mc C$. Therefore, summing over all $\mc C \in R_K$ proves this corollary.
\end{proof}

\begin{rmk}
    In light of this corollary, we will use the notation $\rho_{\mf B}$ instead of $\rho_{\mf B}^{\mc C}$.
\end{rmk}

\section{Formula for Steinitz class}
\label{FS}
In this section, we prove the formula for the Steinitz class given in Proposition \ref{Steinitz Formula Prop}. We determine the formula for the Steinitz class using the following theorem.

\begin{thm}[Artin \cite{Artin}]
\label{Steinitz Computation}
Let $\delta_{L/K}$ be the discriminant of the relative trace form of $L$ over $K$. Then there exists a fractional ideal $\mf a$ of $K$ such that 
    \begin{equation}
    \label{Steinitz Class Equation}
        \Delta_{L/K} = \delta_{L/K}\mf a^2
    \end{equation}
and $[\mf a]$ is the Steinitz class of $L$ over $K$.
\end{thm}

By this theorem, all that we have to do is compute $\Delta_{L/K}$ and $\delta_{L/K}$. We have already stated the formula for $\Delta_{L/K}$ from \cite{D} in Section \ref{PK}. Thus, we compute $\delta_{L/K}$ now.

\begin{prop}
    Let $\ell$ be a prime and $K$ be a number field containing the $\ell$-roots of unity. Then for any $\ell$-Kummer extension $L = K(\sqrt[\ell]{\gamma})$ for $\gamma \in \O_K$, we have $\delta_{L/K} = -\ell^\ell\gamma^{\ell-1}$ for $\ell \neq 2$ and $\delta_{L/K} = \ell^\ell\gamma^{\ell-1}$ for $\ell = 2$.
\end{prop}

\begin{proof}
First, note that
\begin{equation}
    \label{KummerT}
    \Tr_{L/K}(\sqrt[\ell]\gamma) = \sqrt[\ell]\gamma(1+\zeta_\ell + \cdots+\zeta_\ell^{\ell-1}) = 0.
\end{equation}
Since $\ell$ is prime, it follows from \eqref{KummerT} that $\Tr_{L/K}(\sqrt[\ell]{\gamma^n}) = 0$ if $\ell \nmid n$. Therefore, in the relative trace form calculation, we end up with
\begin{align}
    \delta_{L/K} &= \det\begin{pmatrix}
    \Tr_{L/K}(1) &  \Tr_{L/K}(\sqrt[\ell]{\gamma}) & \cdots& \Tr_{L/K}(\sqrt[\ell]{\gamma^{\ell-2}})& \Tr_{L/K}(\sqrt[\ell]{\gamma^{\ell-1}} )\\
    \Tr_{L/K}(\sqrt[\ell]{\gamma} )& \Tr_{L/K}(\sqrt[\ell]{\gamma^{2}} ) & \cdots & \Tr_{L/K}(\sqrt[\ell]{\gamma^{\ell-1}} ) & \Tr_{L/K}(\gamma)\\
    \Tr_{L/K}(\sqrt[\ell]{\gamma^{2}} ) & \Tr_{L/K}(\sqrt[\ell]{\gamma^{3}} ) & \cdots & \Tr_{L/K}(\gamma) & \Tr_{L/K}(\sqrt[\ell]{\gamma^{\ell+1}} )\\
    \vdots & \vdots & \iddots & \vdots & \vdots\\
    \Tr_{L/K}(\sqrt[\ell]{\gamma^{\ell-1}} ) & \Tr_{L/K}(\gamma) & \cdots & \Tr_{L/K}(\sqrt[\ell]{\gamma^{2\ell-3}} ) & \Tr_{L/K}(\sqrt[\ell]{\gamma^{2\ell-2}} )\\
    \end{pmatrix} \\ &=\det\begin{pmatrix}
    \Tr_{L/K}(1) & 0 & \cdots& 0 & 0\\
    0 & 0 & \cdots & 0 & \Tr_{L/K}(\gamma)\\
    0 & 0 & \cdots & \Tr_{L/K}(\gamma) & 0\\
    \vdots & \vdots & \iddots & \vdots & \vdots\\
    0 & \Tr_{L/K}(\gamma) & \cdots & 0 & 0\\
    \end{pmatrix} \\ &= \det\begin{pmatrix}
    \ell & 0 & \cdots& 0 & 0\\
    0 & 0 & \cdots & 0 & \ell\gamma\\
    0 & 0 & \cdots & \ell\gamma & 0\\
    \vdots & \vdots & \iddots & \vdots & \vdots\\
    0 & \ell\gamma & \cdots & 0 & 0\\
    \end{pmatrix} = -\ell^{\ell}\gamma^{\ell-1} 
    \label{Kummer Trace}
\end{align}
for $\ell$ odd. For $\ell = 2$, the sign is positive.
\end{proof}

Combining this result, Theorem \ref{Daberkow General} and Theorem \ref{Steinitz Computation}, we get the formula for the Steinitz class of an $\ell$-Kummer extension stated in Proposition \ref{Steinitz Formula Prop}.

\begin{proof}[Proof of Proposition \ref{Steinitz Formula Prop}]
Let $\gamma$, $\mf Q$, $\mf R$, and $\mf I_i$ for $1 \leq i \leq \ell-1$ be as in the statement of the proposition. By definition, $\mf R^\ell \mf I = \prod_{i=1}^{\ell-1} \mf I_i^i$ is the prime-to-$\ell$-factor of $\gamma\O_K$ decomposed into $i$-power-parts. Let $\mf a$ be the ideal in Equation \eqref{Steinitz Class Equation}. By Equation \eqref{Kummer Trace}, and Theorem \ref{Daberkow General}, we get the string of equalities
$$\mf a^2 = \frac{\Delta_{L/K}}{\delta_{L/K}} =\mf L_{L/K}\left(\frac{\prod_{i=1}^{\ell-1} \mf I_i}{\gamma}\right)^{\ell-1} = \frac{\mf L_{L/K}}{\left(\mf Q\mf R^\ell \prod_{i=2}^{\ell-1} \mf I_i^{i-1}\right)^{\ell-1}}.$$
We can ignore the $-\ell^\ell$ factor in Equation \eqref{Kummer Trace} because $\ell\O_K$ is an $\ell-1$ power of a principal ideal in $\O_K$ as $K$ contains the $\ell$-roots of unity. Therefore, $\ell^{\ell}\O_K$ is the square of a principal ideal, which means it does not affect the Steinitz class computation.
\end{proof}

Before we move on, we should note the following corollary.

\begin{cor}
Let $\ell$ be an odd prime and $K$ be a number field containing the $\ell$-roots of unity. Then for any $\ell$-Kummer extension, the realizable ideal classes of $\ell$-Kummer extensions of $K$ are in $R_K$.
\end{cor}

\begin{proof}
    From Proposition \ref{Steinitz Formula Prop} and the fact that $\mf L_{L/K}$ is an $\ell-1$ power as described in Theorem \ref{Daberkow General}, we see that the Steinitz class is an $\frac{\ell-1}{2}$ power.
\end{proof}

\begin{rmk}
    It follows that not all ideal classes are realizable as Steinitz classes of an $\ell$-Kummer extension for $\ell \geq 5$, a fact shown by McCulloh \cite{Mc}.
\end{rmk}

\begin{rmk}
    For $\ell=2$, the expression for Steinitz class becomes
    $$\mf a^2 = \frac{\mf L_{L/K}}{\mf Q\mf R^2}.$$
    It is less obvious that this right-hand-side is a square in this case. However, by Theorem \ref{Daberkow General}, $\frac{\mf L_{L/K}}{\mf Q}$ must be a square.
\end{rmk}

\section{Equidistribution of Steinitz Classes}
\label{KE}
At last, we prove our equidistribution theorem. 

\begin{proof}[Proof of Theorem \ref{Main Theorem}]
Following the convention of \cite{CohenDisc}, we will count extensions bounded by $X^{\ell-1}$. By Corollary \ref{Choices}, an $\ell$-Kummer extension corresponds to four choices under the map $\Phi$

. We count them as follows:
\begin{enumerate}
    \item By Dirichlet's unit theorem and the fact that $K$ contains the $\ell$-th roots of unity, $\#U/U^{\ell} = \ell^{r_1+r_2}$.

    \item  By Lemma \ref{Reduction General}, it is sufficient to pick an ideal class $\mc R \in \Cl(K)$ and a representative integral ideal $\mf R$ from that class. Therefore, there are $h_K$ possibilities for this choice.

    \item  We use $Q$ to denote the set of $\ell$-power-free ideals with all prime factors dividing $\ell\O_K$. Lemma \ref{Reduction General} allows us to assume that the $\ell$-factor of $\gamma\O_K$ is $\ell$-power-free. Therefore, there are $\# Q$ possibilities for this choice.

    \item By Corollary \ref{Equidistribution of Radicals}, the number possible $\ell$-power-free parts $\mf I$ that make $\gamma\O_K$ principal, which is to say
    $$\prod_{i=1}^{\ell-1} \mf I_i^{i} \in \left[\mf q^m \mf Q\mf R^{\ell} \right]^{-1},$$ 
    is given by $\frac{\lambda_{\ell-1}}{h_K} X\log^{\ell-2} X + O(X\log^{\ell-3} X)$ where 
    $$\lambda_{\ell-1} = \frac{\left(\displaystyle\res_{s=1}\zeta_K(s)\right)^{\ell-1}}{(\ell-2)!} \left[\prod_{(\mf p,\ell\O_K)=\O_K} \left(1+\frac{\ell-1}{N(\mf p)}\right)\left(1-\frac 1 {N(\mf p)}\right)^{\ell-1}\right]\left[\prod_{\mf q \mid \ell\O_K} \left(1-\frac{1}{N(\mf q)}\right)^{\ell-1}\right].$$
    This is independent of the previous choices and of the specified $\mc C$.

\end{enumerate}
    
    As before, not all of these extensions we have counted will have the $\mf L_{L/K} = \mf B$. The portion of those that do is given by $\rho_{\mf B}$, which we showed was independent of Steinitz class. Also, we divide by $\ell-1$ because $\Phi$ is an $(\ell-1)$-to-one map. Putting this together gives  
    \begin{align*}
    \#E_{\mf B}(X^{\ell-1}) &= \#\{[L:K] = \ell : N(\mf F_{L/K}) \leq \frac{X^{\ell-1}}{N(\mf B)},\ \mf L_{L/K} = \mf B\}\\
    &= \frac{\rho_{\mf B}}{\ell-1}\#U/U^\ell\sum_{\mc R \in \Cl(K)}\sum_{\mf Q \in Q} \frac{\lambda_{\ell-1}}{h_K} \frac{X}{\sqrt[\ell-1]{N(\mf B)}}\log^{\ell-2} X + O(X\log^{\ell-3} X)\\
    &= \frac{\rho_{\mf B}\ell^{r_1+r_2} \#Q \lambda_{\ell-1}}{\ell-1} \frac{X}{\sqrt[\ell-1]{N(\mf B)}}\log^{\ell-2} X + O(X\log^{\ell-3} X).
    \end{align*}
   
    To specify a Steinitz class $\mc C$ for $\ell > 2$, we require that for $\mf I = \prod_{i=1}^{\ell-1} \mf I_i^i$, we have both
    $$\prod_{i=1}^{\ell-1} \mf I_i^{i} \in \left[\mf q^m \mf Q\mf R^{\ell} \right]^{-1}$$ 
    and
   $$\left(\prod_{i=2}^{\ell-1} \mf I_i^{i-1}\right)^{\frac{\ell-1}{2}} \in \mc C^{-1}\left[\sqrt{\frac{\mf B}{(\mf R^{\ell}\mf Q)^{\ell-1}}}\ \right].$$ 
    Therefore, we use Corollary \ref{Equidistribution of Radicals 2} instead of Corollary \ref{Equidistribution of Radicals} when counting $\ell$-power-free parts. This gives a very similar computation 
    \begin{align*}
    \#E_{\mf B}^{\mc C}(X^{\ell-1}) &= \#\{[L:K] = \ell : N(\mf F_{L/K}) \leq \frac{X^{\ell-1}}{N(\mf B)},\ \mf L_{L/K} = \mf B, \St(\O_L) = \mc C\}\\
   &= \frac{\rho_{\mf B}}{\ell-1}\#U/U^\ell\sum_{\mc R \in \Cl(K)}\sum_{\mf Q \in Q} \frac{\lambda_{\ell-1}}{h_K \#R_K} \frac{X}{\sqrt[\ell-1]{N(\mf B)}}\log^{\ell-2} X + O(X\log^{\ell-3} X)\\
    &= \frac{\rho_{\mf B}\ell^{r_1+r_2} \#Q \lambda_{\ell-1}}{\#R_K(\ell-1)} \frac{X}{\sqrt[\ell-1]{N(\mf B)}}\log^{\ell-2} X + O(X\log^{\ell-3} X)\\
    &= \frac 1 {\#R_K} \#E_{\mf B}(X^{\ell-1}) + O(X\log^{\ell-3} X).
    \end{align*}
    
    For $\ell=2$, the Steinitz class is determined by $\mc R \in \Cl(K)$. Specifically, we have that 
    $$\mc R \in \mc C^{-1}\left[\sqrt{\frac{\mf B}{\mf Q}}\right].$$
    Therefore, we have
    \begin{align*}
    \#E_{\mf B}^{\mc C}(X) &= \#\{[L:K] = 2 : N(\mf F_{L/K}) \leq \frac{X}{N(\mf B)},\ \mf L_{L/K} = \mf B, \St(\O_L) = \mc C\}\\
   &= \rho_{\mf B}\#U/U^2\sum_{\mf Q \in Q} \frac{\lambda_{1}}{h_K} \frac{X}{N(\mf B)}+ O\left(\frac X{\log X}\right)\\
    &= \frac{\rho_{\mf B}2^{r_1+r_2} \#Q \lambda_{1}}{h_KN(\mf B)}  X + O\left(\frac X{\log X}\right)\\
    &= \frac 1 {h_K} \#E_{\mf B}(X) + O\left(\frac X{\log X}\right).
    \end{align*}
    Recall that $h_K = \#R_K$ for $\ell=2$.
    For both cases, summing over all $\mf B \in B$ completes the proof.
    \end{proof}

Leaving out the Steinitz class condition, we can sum $E_{\mf B}(X^{\ell-1})$ over $B$ as well to enumerate $E(X^{\ell-1})$.

\begin{cor}
\label{KEnum}
Let all notation be defined as above, then
    $$\#E(X^{\ell-1}) = \frac{\ell^{r_1+r_2} \#Q \lambda_{\ell-1}}{\ell-1} \left(\sum_{\mf B \in B} \frac{\rho_\mf B}{\sqrt[\ell-1]{N(\mf B)}}\right) X\log^{\ell-2} X  + O(X\log^{\ell-3} X).$$
\end{cor}

Results of this type are not new. Using class field theory, Cohen, Diaz y Diaz, and Olivier proved a general result for cyclic extensions of prime degree of which the following is a special case \cite{CohenDisc}.

\begin{thm}[Cohen-Diaz y Diaz-Olivier]
\label{CohenK}
    Let all notation be defined as above, then
    $$\#E(X^{\ell-1}) \sim \frac{\lambda_{\ell-1}'}{\ell^{r_2}} X\log^{\ell-2} X$$
where
$$\lambda_{\ell-1}' = \frac{\lambda_{\ell-1}}{\ell-1}\prod_{\mf q \mid \ell\O_K} \left(1+\frac{\ell-1}{N(\mf q)}\right).$$
\end{thm}

Our version is equivalent but slightly more complicated. Equating asymptotic constants in Corollary \ref{KEnum} and Theorem \ref{CohenK}, gives a new and interesting formula.

\begin{cor}
    Let all notation be as above, then
    $$\sum_{\mf B \in B} \frac{\rho_{\mf B}}{\sqrt[\ell-1]{N(\mf B)}} = \frac{1}{\ell^{r_1+2r_2}\#Q}\prod_{\mf q \mid \ell\O_K} \left(1+\frac{\ell-1}{N(\mf q)}\right).$$
\end{cor}

\begin{ex}
We consider this formula in the case that $\ell = 2$ and $K = \Q[x]/(x^3-x-9)$. Note that 2 is inert in $\O_K$ and $r_1 = r_2 = 1$. From Theorem \ref{Daberkow General}, the possible even parts of the discriminant are 1, 4 and 8. As shown in a previous section, $\rho_8 = \frac 1 2$. We can also find that $\rho_4 = \frac 7 {16}$ and $\rho_1 = \frac {1}{16}$ because 
$(\O_K/4\O_K)^\times \cong \Z/7\Z \times (\Z/2\Z)^3$ so $\frac 1 8$ of odd elements of $\O_K$ are squares mod 4. Additionally, $\#Q = 2$ because 1 and 2 are the only squarefree divisors of 2 in $\O_K$. We compute
\begin{align*}
\sum_{\mf B \in B} \left(\frac{\rho_{\mf B}}{N(\mf B)}\right) &= \left(\frac 1 {16} + \frac 7 {1024} + \frac 1 {1024}\right) 
=\frac 9 {128}\end{align*}
and
\begin{align*}
\frac 1 {(2^{r_1+2r_2}\#Q)}\left(\prod_{\substack{\mf q \mid 2\O_K\\ 
\mf q \text{ prime}}}1+\frac{1}{N(\mf q)}\right) = 
\frac 1 {2^3 (2)}\left(1+\frac 1 8\right) = \frac 9 {128}.
\end{align*}

\end{ex}

\section{Acknowledgments}
The author thanks Professor Siman Wong for his guidance throughout the process of writing this paper. They also thank Professor Adebisi Agboola for pointing out a typographical error in the introduction and Professor Anthony Kable for suggesting a simpler proof of Lemma \ref{Group product distribution}. In addition, they thank the anonymous referee for many helpful suggestions for improving the whole paper, and Section \ref{PK} in particular.

\section{Notation Index}
\begin{center}
\begin{longtable}{c|c|c} 
 Notation & Page & Definition\\
 \hline
 $N$ & \pageref{N} & The absolute value of $N_{K/\Q}$\\
 $R_K$ & \pageref{Main Theorem} & $R_K = \begin{cases} \text{Subgroup of } \left(\frac{\ell-1}{2}\right)\text{-powers in } \Cl(K)& \ell \neq 2\\ \Cl(K) & \ell = 2\end{cases}$\\
$[\mf A]$ & \pageref{Steinitz Formula Prop} & The ideal class of the ideal $\mf A$\\
 $\mf L_{L/K}$ & \pageref{L} & The $\ell$-factor of $\Delta_{L/K}$\\
 $\mf F_{L/K}$ & \pageref{F} & The prime-to-$\ell$-factor of $\Delta_{L/K}$\\
 $B$ & \pageref{L} & The set of possible $\mf L_{L/K}$ for $\ell$-Kummer extensions of $K$\\
 $U$ & \pageref{U} & The set of units in $\O_K$\\
 $\mf Q$ & \pageref{Steinitz Formula Prop} & The $\ell$-factor of $\gamma\O_K$\\
 $\mf R^{\ell}$ & \pageref{Steinitz Formula Prop} & The $\ell$-power-part of $\gamma\O_K$\\
 $\mf B$ & \pageref{mfQ} & Used to represent an element of $B$\\
 $\#E(X)$ & \pageref{E} & $\{L/K : \Gal(L/K) \cong \Z/\ell\Z,\ N(\Delta_{L/K}) \leq X\}$\\
$E^{\mc C}(X)$ & \pageref{E} & $\{L/K \in E(X) : \St(\O_L) = \mc C\}$\\
$E_{\mf B}(X)$ & \pageref{E} & $\{L/K \in E(X) : \mf L_{L/K} = \mf B\}$\\
$E_{\mf B}^{\mc C}(X)$ & \pageref{E} &  $E_{\mf B}(X) \cap E^{\mc C}(X)$\\
$\rho_{\mf B}^{\mc C}$ & \pageref{rho_q} & $\displaystyle\lim_{X\rightarrow \infty} \frac{\#\{L/K: \Gal(L/K) \cong \Z/\ell\Z,\ N(\mf F_{L/K}) \leq X^{\ell-1},\ \St(\O_L) = \mc C,\ \mf L_{L/K} = \mf B\}}{\#\{L/K: \Gal(L/K) \cong \Z/\ell\Z,\  N(\mf F_{L/K}) \leq X^{\ell-1},\ \St(\O_L) = \mc C\}}$\\
$\lf(\mf A)$ & \pageref{f} & The $\ell$-power-free-factor of the ideal $\mf A$\\
$P$ & \pageref{P} & Non-zero principal ideals in $\O_K$\\
$F$ & \pageref{P} & A fixed (non-canonical) map from $P \times U$ to $\O_K \wo \{0\}$\\
 $Q$ & \pageref{Choices} &  The set of $\ell$-power-free ideals with all prime factors dividing $\ell\O_K$ \\
 $J$ & \pageref{Choices} & The set of $\ell$-power-free, prime-to-$\ell$ ideals of $\O_K$\\
 $\Phi$ & \pageref{Choices} & $\Phi(u,\mc R, \mf Q,\mf I)$ gives an $\ell$-Kummer extension for principal $\mc R^{\ell}\mf Q\mf I$\\
$W$ & \pageref{W} & $\{\mf q : \ell \O_K \mid \mf q \text{ prime, } \nu_{\mf q}(\mf B) = \ell-1 + \ell\nu_{\mf q}(\ell)\}$\\ 
$W'$ & \pageref{W} & $\{\mf q : \ell \O_K \mid \mf q \text{ prime, } \nu_{\mf q}(\mf B) \neq \ell-1 + \ell\nu_{\mf q}(\ell)\}.$\\
$\mf W$ & \pageref{W} & $\prod_{\mf q \in W'} \mf q^{\ell\nu_{\mf q}(\ell)}$\\
$G$ & \pageref{GH}&  $(\O_K/\mf W)^\times$\\
$H$ & \pageref{GH} & $G/G^{\ell}$\\
$\O_K(\mf W)$ & \pageref{GH} & The set of $\gamma \in \O_K$ that are relatively prime to $\mf W$.\\
$\psi$ & \pageref{GH} & Composition of quotient maps from $\O_K(\mf W)$ to $H$.\\
$V$ & \pageref{V} & The free part of $U$\\
$T$ & \pageref{V} & Multiplicative semi-group $\O_K \wo \{0\}$\\
$Y$ & \pageref{V} & Set of cosets of $S/V^{\ell}$\\
$(r_1,r_2)$ & \pageref{D} & The number of real and complex pairs of embeddings\\
$M$ & \pageref{D} & Minkowski embedding\\
$D$ & \pageref{D} & Fundamental domain for $M(B)$\\
$\lambda_k$ & \pageref{Divisors} & $\displaystyle\frac{\left(\displaystyle\res_{s=1}\zeta_K(s)\right)^{k}}{(k-1)!} \left[\prod_{(\mf p,\ell\O_K)=\O_K} \left(1+\frac{k}{N(\mf p)}\right)\left(1-\frac 1 {N(\mf p)}\right)^{k}\right]\left[\prod_{\mf q \mid \ell\O_K} \left(1-\frac{1}{N(\mf q)}\right)^{k}\right]$\\
$\delta_{L/K}$ & \pageref{Steinitz Computation} & The discriminant of the relative trace form of $L/K$\\
\end{longtable}
\end{center}

\bibliography{Bibliography}

@article{KW, 
title={Uniform distribution of the Steinitz invariants of quadratic and cubic extensions}, 
volume={142}, 
number={1}, 
journal={Compositio Mathematica}, 
publisher={London Mathematical Society}, 
author={Kable, Anthony C. and Wright, David J.}, 
year={2006}, 
pages={84–100}}

@book{M,
  title={Number Fields},
  author={Marcus, D.A.},
  isbn={9783540902799},
  lccn={77021467},
  year={1977},
  publisher={Springer-Verlag}
}

@article{Artin,
    author = {Emil Artin},
    title = {Questions de base minimale dans la theoire des nombres algebriques},
    journal = {International conferences of the National Center for Scientific Research: Algebra and Number Theory},
    year = {1949}
}

@book{L,
    author = {Serge Lang},
    title = {Algebraic Number Theory},
    year = {1994}, 
    publisher = {Springer New York}
}

@book{Cohen,
  title={Advanced topics in computational number theory},
  author={Henri Cohen},
  year={2000},
  publisher = {Springer New York}
}

@article{D,
title = {On Computations in Kummer Extensions},
journal = {Journal of Symbolic Computation},
volume = {31},
number = {1},
pages = {113-131},
year = {2001},
issn = {0747-7171},
author = {Mario Daberkow}}

@article{Mc,
 ISSN = {00029939, 10886826},
 URL = {http://www.jstor.org/stable/2036117},
 author = {Leon R. McCulloh},
 journal = {Proceedings of the American Mathematical Society},
 number = {5},
 pages = {1191--1194},
 publisher = {American Mathematical Society},
 title = {Cyclic Extensions Without Relative Integral Bases},
 urldate = {2024-05-25},
 volume = {17},
 year = {1966}
}

@phdthesis{Foster,
    author ={Kurt Foster},
    title = {An equal-distribution result for galois module structure},
    school = {University of Illinois Urbana Champaign},
    year = {1987}
}

@article{CohenDisc,
author = {Cohen, Henri and Diaz y Diaz, Francisco and Olivier, Michel},
title = {On the density of discriminants of cyclic extensions of prime degree},
pages = {169--209},
volume = {2002},
number = {550},
journal = {Journal für die reine und angewandte Mathematik},
year = {2002},
}

@unpublished{Bhargava,
author = {Bhargava, Manjul and Shankar, Arul and Wang, Xiaoheng},
year = {2015},
title = {Geometry-of-numbers methods over global fields I: Prehomogeneous vector spaces},
note = {\url{https://arxiv.org/abs/1512.03035}}, 
}

@article{BQuart,
 author = {Manjul Bhargava},
 journal = {Annals of Mathematics},
 number = {2},
 pages = {1031--1063},
 publisher = {Annals of Mathematics},
 title = {The Density of Discriminants of Quartic Rings and Fields},
 volume = {162},
 year = {2005}
}

@book{LangA,
    author = {Serge Lang},
    title = {Algebra},
    publisher = {Springer},
    year = {2003}
}

@article{J,
author = {Jakimczuk, Rafael and Lalín, Matilde},
year = {2022},
month = {10},
pages = {617-634},
title = {Asymptotics of sums of divisor functions over sequences with restricted factorization structure},
volume = {28},
journal = {Notes on Number Theory and Discrete Mathematics},
doi = {10.7546/nntdm.2022.28.4.617-634}
}

@article{Agboola,
    title={On counting rings of integers as Galois modules}, 
    author={A. Agboola},
    journal = {Journal f\"ur die reine und angewandte Mathematik},
    year={2012},
    publisher = {De Gruyter},
    volume = {663},
    pages = {1-31}
}

@book{T,
    title = {Introduction to Analytic and Probabilistic Number Theory},
    author = {G\'erald Tenenbaum},
    publisher = {American Mathematical Society},
    year = {2008}}

@book{N,
    title = {Elementary and Analytic Theory of Algebraic Numbers},
    author = {W. Narkiewicz},
    publisher = {Springer},
    year = {2004}}

@article{Bekyel,
author = {Ebru Bekyel},
title = {Steinitz classes and discriminant counting},
journal = {Acta Arithmetica},
volume = {118},
pages = {27-40},
number = {1},
year = {2005}}

@unpublished{GW,
      title={Counting biquadratic number fields with quaternionic and dihedral extensions}, 
      author={Louis M. Gaudet and Siman Wong},
      year={2025},
      note = {\url{https://arxiv.org/abs/2506.21522}}}

@article{Murty,
    author = {M.R. Murty, J. Sahoo and A. Vatwani} ,
    title = {A simple proof of the Wiener-Ikehara Tauberian Theorem},
    journal = {Expositiones Mathematicae},
    year = {2024},
    volume = {42}
}
\bibliographystyle{plain}
\end{document}